\documentclass[11pt]{article}

\usepackage{amssymb,amsmath, amsthm}
\usepackage[width=17cm,height=21cm]{geometry}
\usepackage{txfonts}
\usepackage{textcomp, gensymb}
\usepackage[utf8]{inputenc}
\usepackage[title]{appendix}
\usepackage{color}
\usepackage{graphicx}
\usepackage{pdfpages}
\usepackage{hyperref}
\usepackage{mathtools}
\usepackage[capitalise]{cleveref}
\usepackage{nicefrac}
\usepackage{stmaryrd}
\usepackage{authblk}
\usepackage{float}
\usepackage{subcaption}
\usepackage[ruled,vlined]{algorithm2e}
\usepackage[pagewise]{lineno}
\usepackage{diagbox}

\usepackage{changes}

\newcommand{\R}{\mathbb{R}}
\newcommand{\Z}{\mathbb{Z}}
\newcommand{\e}{\varepsilon}
\newcommand{\di}[1]{\,\mathrm{d}#1}
\newcommand{\dive}{\operatorname{div}}
\newcommand{\interior}{\operatorname{int}}
\newcommand{\ddt}{\frac{\operatorname{d}}{\operatorname{d}t}}
\newcommand{\twosc}{\stackrel{2}{\rightharpoonup}}

\newcommand{\jump}[1]{\llbracket #1\rrbracket}
\newcommand{\thetaIdx}[1]{\theta^{#1}_\varepsilon}
\newcommand{\OmegaIdx}[1]{\Omega^{#1}_\varepsilon}
\newcommand{\alphafs}{\alpha^{0}}
\newcommand{\alphafg}{\alpha^{f}}
\newcommand{\alphasg}{\alpha^{s}}
\newcommand{\Sigmafs}{\Sigma^{0}_\varepsilon}
\newcommand{\Sigmafg}{\Sigma^{f}_\varepsilon}
\newcommand{\Sigmasg}{\Sigma^{s}_\varepsilon}
\newcommand{\Gammafs}{\Gamma^{0}}
\newcommand{\Gammafg}{\Gamma^{f}}
\newcommand{\Gammasg}{\Gamma^{s}}
\newcommand{\GrainCell}{Z}

\newcommand\restr[2]{{
  \left.\kern-\nulldelimiterspace 
  #1 
  \vphantom{\big|} 
  \right|_{#2} 
  }}

\makeatletter
\newcommand{\myitem}[1]{%
\item[#1]\protected@edef\@currentlabel{#1}%
}
\makeatother

\newtheorem{definition}{Definiton}
\newtheorem{theorem}{Theorem}

\newtheorem{lemma}{Lemma}
\newtheorem{remark}{Remark}

\crefname{lemma}{Lemma}{lemmas}
\Crefname{lemma}{Lemma}{Lemmas}
\crefname{thm}{theorem}{theorems}
\Crefname{thm}{Theorem}{Theorems}
\Crefname{algocf}{Algorithm}{Algorithm}

\numberwithin{equation}{section}

\begin{document}

\title{Homogenization and simulation 
      of heat transfer through a thin grain layer}

\author[$\star, 1$]{Tom Freudenberg}
\author[$\dagger$]{Michael Eden}

\affil[$\star$]{Center for Industrial Mathematics, University of Bremen, Germany}
\affil[$\dagger$]{Department of Mathematics and Computer Science, Karlstad University, Sweden}

\maketitle

\begin{abstract}
We investigated the effective influence of grain structures on the heat transfer between a fluid and solid domain using mathematical homogenization. 
The presented model consists of heat equations inside the different domains, coupled through either perfect or imperfect thermal contact.  
The size and the period of the grains are of order $\varepsilon$, therefore forming a thin layer. The equation parameters inside the grains also depend on $\varepsilon$. 
We considered two distinct scenarios: Case (a), where the grains are disconnected, and Case (b), where the grains form a connected geometry but in a way such that the fluid and solid are still in contact. 
In both cases, we determined the effective differential equations for the limit $\varepsilon \to 0$ via the concept of two-scale convergence for thin layers. 
We also presented and studied a numerical algorithm to solve the homogenized problem.
\end{abstract}


\setcounter{footnote}{1} 
\footnotetext{Corresponding author, Email: \hyperlink{tomfre@uni-bremen.de}{tomfre@uni-bremen.de}}

\section{Introduction}
We consider the heat dynamics in a domain $\Omega\subset\R^d$ consisting of a fluid region $\OmegaIdx{f}$ and an adjacent solid domain $\OmegaIdx{s}$, where small grain structures $\OmegaIdx{g}$ are periodically distributed along the fluid-solid interface $\Sigma$.
The height as well as the period of these grain structures is denoted by $\varepsilon > 0$, which is assumed to be much smaller than the overall size of the solid-fluid system $\Omega$.
This results in a $\e$-sized layer region in which all three regions (fluid, solid, and grains) are in contact.
Our objective is to determine the effective model for $\varepsilon \to 0$ via mathematical homogenization for thin domains.
Here, we consider both perfect thermal contact and imperfect heat exchange, modeled via a Robin condition, between the different regions.  

In the following, we assume that there is, for all $\varepsilon > 0$, direct contact between the fluid and the solid region. Regarding the grain geometries, we distinguish between two different cases:
\begin{itemize}
    \item \textbf{Case (a)}: Disconnected grains that are periodically distributed along the interface $\Sigma$.
    For the effective model, we obtain a two-scale problem with microstructures at the interface, similar to the homogenization results in \cite{Eden2022effective, Eden22, SHOWALTER04}.
    \item \textbf{Case (b)}: Connected grains, comparable to a sieve between the two regions. Here, we arrive at a Wentzell-Robin interface temperature \cite{Arendt03, Bonnailllie10} in the effective model, similarly to \cite{Amirat06, Bendali96, Gahn17}. 
\end{itemize}

Our research, particularly Case (a), is motivated by the influence and interplay of cooling fluids with the grain structure in a grinding process.
The friction between the grains and the workpiece heats the system.
Here, the grain can locally reach temperatures of up to 1200\,\degree C, which is much higher than the temperature of the surrounding region \cite{Rowe03, UEDA93, Wiesener22}.
On that point, it is important to accurately capture the impact the grains have on the temperature distribution to determine effects like workpiece burn and wear of the grinding wheel.
Current models only include the heat produced by the grains and do not include the grains directly in their simulations \cite{Chen20, GU04, YANG19}, mainly because of the high computational cost of simulating the whole grinding wheel with resolved grains.
Even if our model is somewhat idealized with periodically distributed grains and without direct consideration of the workpiece, we provide a base for including more details of grain geometry in future simulations. 
Some relevant early experimental results and simulations for this specific scenario, where individual grains were taken into account, can be found in \cite{Wiesener22}.

Other possible applications are reaction-diffusion systems where grain structures play a role. For example, a diffusion problem in a riverbed where our solid bulk would be replaced with porous media and the grains could represent larger rock formations at the ground of the river \cite{NAGAOKA90}.
Case (b) may be useful for filtering problems, particularly if one replaces the solid domain $\OmegaIdx{s}$ with an additional fluid region. Our results could then be combined with already established homogenization models for Stokes flow through thin filters \cite{Allaire91, Conca88}.  

To derive the effective model, we apply the concept of two-scale convergence for thin heterogeneous layers, first introduced in \cite{Neuss07}.
Similar problems for diffusion equations were considered in \cite{Gahn17, Gahn18}.
One novel aspect in our research is that the fluid domain $\OmegaIdx{f}$ and the solid domain $\OmegaIdx{s}$ are in direct contact and not completely isolated by the grain structure.
This leads to additional coupling conditions between fluid and solid as well as a slight modification of the two-scale concept, whereby the already established two-scale theory and results can be transferred to our scenario.
Additionally, we derive an $\e$-independent trace estimate for domains that have a rough boundary given by a finite union of height functions.
Next to the analysis, we also carry out numerical investigations where we face the challenge of coupling the solution on the macro domain with the solution of the cell problems.
This is handled with an iterative algorithm, similar to \cite{Eden2022effective}.
We study the influence of a relaxation scheme on the number of iterations and also investigate numerically the limit behavior for $\e \to 0$.

Comparable to our problem is also the case of diffusion through fast oscillating interfaces with small \cite{Donato19, Donato10} or fixed \cite{Nevard97} amplitude, which was already extensively studied. Comparable differential equations and a similar geometrical setup, but without the interface grains and a porous media instead of the solid, were analyzed in our previous work \cite{Eden2022effective}.  
Additionally, the improved heat exchange at rough boundaries, for example between a fluid and a wall with a fixed temperature, is widely studied in the literature since it is a useful property in many applications \cite{Bhavnani1991, Owen63}. The research ranges from pure numerical studies \cite{Ting19, Yousaf15} to multiple-scale expansion and asymptotic matching \cite{ahmed22, INTROINI11}. While these studies are generally only concerned with heat transfer into the fluid, we want to study the cooling effect on the adjacent solid and rigorously derive an effective model for a thin layer of grains.

This paper is structured as follows: In \cref{sec:setup}, we introduce the mathematical model, including information about the studied geometries, the assumptions to pass to the limit $\varepsilon \to 0$, and possible limitations of the model. \cref{sec:analysis} handles the analysis of the present microscale model and we derive solution bounds for the microscale problem. In \cref{sec:homogenization}, the detailed homogenization procedure is presented and the homogenized models are stated and analyzed. Finally, in \cref{sec:simulations}, we introduce an algorithm for the effective model and demonstrate various simulation results.  
\section{Setup, notation, and mathematical equation}\label{sec:setup}
\subsection{Description of the geometry}
We start by introducing the geometric setup and notation as well as the mathematical models considered in this work.
After that, we cover the assumptions needed to apply homogenization to our problem.
In the following, the geometry and the model are split into three subdomains representing the fluid part, the solid part, and the grains.
Functions, parameters, and subdomains are denoted by the corresponding superscripts $f$, $s$, $g$.

The time interval is denoted by $S=(0, T)$, for $T>0$. The spatial domain $\Omega \subset \R^d$ is a cylinder given by $\Omega= \Tilde{\Omega} \times (-H, H)$, with a bounded Lipschitz domain $\Tilde{\Omega} \subset \R^{d-1}$ and height $H > 0$.
For the homogenization, we require that $\Tilde{\Omega}$ be perfectly tiled with axis-parallel $(d-1)$-dimensional cubes with corner coordinates in $\varepsilon_0\mathbb{Z}^d$ for some $\varepsilon_0 > 0$.
We define the subdomains and interface
\begin{linenomath*}\begin{equation*}
    \Omega^f = \Tilde{\Omega} \times (0, H),\quad 
    \Omega^s = \Tilde{\Omega} \times (-H, 0),\quad
    \text{ and } \quad 
    \Sigma = \Tilde{\Omega} \times \{0\}.
\end{equation*}\end{linenomath*}
A point $x \in \Omega$ will also be denoted by $x = (\Tilde{x}, x_d) \in \Tilde{\Omega} \times (-H, H)$.

We denote the $i$-dimensional unit cube by $Y^i=(0, 1)^i$.
The reference grain geometry $\GrainCell \subset Y^{d-1} \times (-1, 1)$ is assumed to be a Lipschitz domain.
For a unified notation, the intersection of the grains with fluid or solid should neither be empty, meaning
\begin{linenomath*}\begin{equation*}
    \{y\in \GrainCell : y_d > 0\} \neq \emptyset \quad  
    \text{ and } \quad 
    \{y\in \GrainCell : y_d < 0\} \neq \emptyset.
\end{equation*}\end{linenomath*}
In addition, the above sets are assumed to be connected Lipschitz domains; in particular, $\GrainCell$ is not allowed to have holes.
The flat surface without the cell and the boundaries of the grain cell are noted by
\begin{linenomath*}\begin{equation*}
    \Gammafs = (Y^{d-1} \times \{0\}) \setminus \GrainCell, 
    \quad \Gammafg = \{y\in \partial \GrainCell : y_d > 0\}
    \quad \text{ and} \quad 
    \quad \Gammasg = \{y\in \partial \GrainCell : y_d < 0\}.
\end{equation*}\end{linenomath*}
In addition to the Lipschitz assumption on the underlying domain $Z$, we assume that the vertical interface section, i.e., the set $\{x\in \partial \GrainCell: n_\Gamma(x)\cdot e_d = 0\}$, has surface measure $0$.
This assumption is needed for an $\e$-independent trace estimate and it allows us to represent the interface as a graph of a finite number of height functions defined over $[0, 1]^{d-1}$; one example is visualized in Figure \ref{fig:interface_example}.
Please note that with this assumption, we exclude, for example, rectangular cuboids.
See the proof of \cref{lem:trace_estimate} for a remark and possible extension to also include general Lipschitz boundaries.

\begin{figure}[H]
	\centering
	\includegraphics[width=0.6\linewidth]{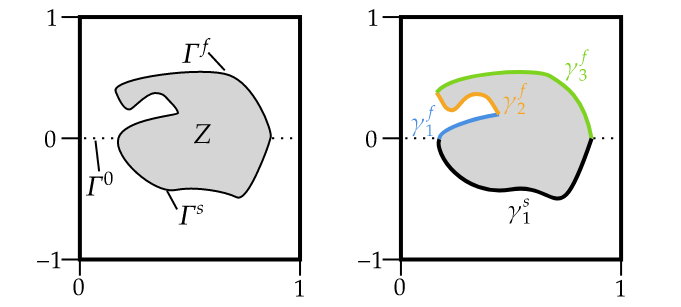}
	\caption{Left: one example of a cell geometry and visualization of the introduced notation. Right: decomposition of the interface into multiple graphs.}
	\label{fig:interface_example}
\end{figure}

The periodic interface structure of grains is created by scaling and tiling the reference cell.
Take $\varepsilon_0 > 0$ such that $\Tilde{\Omega}$ can be perfectly tiled with $\varepsilon_0 Y^{d-1}$ cells.
For a finer, perfect tiling, one can use $\varepsilon_n = \frac{1}{2^n}\varepsilon_0$, where in the following we suppress the index $n$.
At the interface $\Sigma$, we define
\begin{linenomath*}\begin{equation*}
        \OmegaIdx{g} = \text{int}\left(\bigcup_{\Tilde{k}\in\Z^{d-1}} \varepsilon \left(\overline{\GrainCell} + \left(\Tilde{k}, 0\right)\right)\right) \cap \Omega
\end{equation*}\end{linenomath*}
where we also assume $\OmegaIdx{g}$ to be Lipschitz.
For the domain $\OmegaIdx{g}$, two different cases are considered:
\begin{enumerate}
    \item[(a)] $\overline{\GrainCell} \subset Y^{d-1} \times (-1, 1)$.
    Therefore, $\Omega_\varepsilon^g$ is \textbf{disconnected}.
    This case is motivated by the application of grinding wheels where the grains are usually distributed and held together by a binding material.
    \item[(b)] Both $(Y^{d-1} \times (-1, 1)) \setminus \GrainCell$ and $\OmegaIdx{g}$ are \textbf{connected} and for $i=1,\dots,d-1$ it holds
    \begin{equation*}
        \left\{y \in \partial Z : y_i = 1 \right\} = \left\{y \in \partial Z : y_i = 0 \right\} + e_i.
    \end{equation*}
    In this case, the microstructure $\OmegaIdx{g}$ can be viewed as a kind of sieve between fluid and solid.
\end{enumerate}
Both cases and the notation are visualized in Figure \ref{fig:domain_picture}.
\begin{figure}[H]
    \centering
    \includegraphics[width=0.75\linewidth]{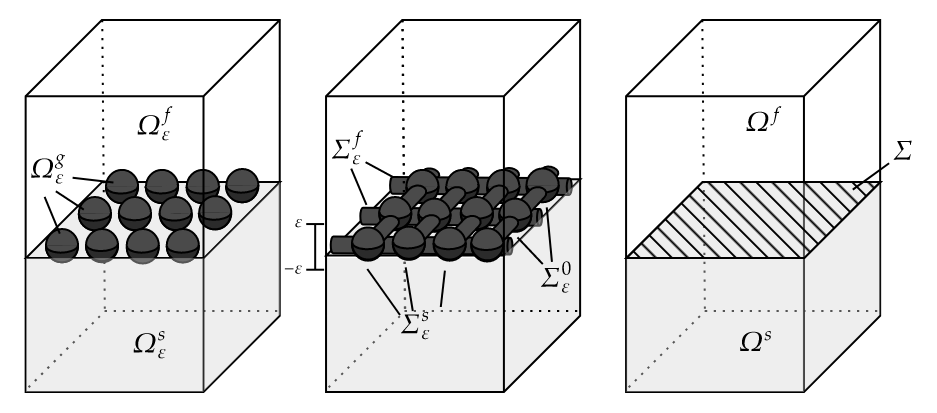}
    \caption{Schematic depiction of the different geometries considered. Left: disconnected grain structure. Center: connected microgeometry. Right: macroscopic domain.}
    \label{fig:domain_picture}
\end{figure}

Last, we obtain the $\varepsilon$-periodic subdomains
\begin{linenomath*}\begin{equation*}
        \Omega_\varepsilon^f = \interior\left(\overline{\Omega^f} \setminus 
        \Omega^g_\varepsilon\right), \quad \Omega_\varepsilon^s = \interior\left(\overline{\Omega^s} \setminus 
        \Omega^g_\varepsilon\right),
        \quad 
        \Omega_\e = \OmegaIdx{f} \cup \OmegaIdx{s}
\end{equation*}\end{linenomath*}
and the three interfaces
\begin{linenomath*}\begin{equation*}
    \Sigmafs = \partial\Omega_\varepsilon^{f} \cap \partial\Omega_\varepsilon^{s}, \quad
    \Sigmafg = \partial\Omega_\varepsilon^{f} \cap \partial\Omega_\varepsilon^{g}, \quad
    \Sigmasg = \partial\Omega_\varepsilon^{s} \cap \partial\Omega_\varepsilon^{g}.
\end{equation*}\end{linenomath*}
Please note that for the volume of the layer, it holds $|\OmegaIdx{g}| \in \mathcal{O}(\varepsilon)$, and for the surface of the interfaces $|\Sigmafg|, |\Sigmasg|, |\Sigmafs| \in \mathcal{O}(1)$.

In the mathematical model, multiple normal vectors of the domain appear. With $\nu=\nu(x)$, we denote the normal vector pointing outwards on $\partial \Omega$. The outward normal in $\Omega$ is denoted by $n=n(x)$, the normal pointing out of $\OmegaIdx{g}$ by 
$n_\varepsilon=n_\varepsilon(x)$, and lastly, the ones out of $\GrainCell$ by $n_\Gamma = n_\Gamma(x)$.  
\subsection{Mathematical model}
We start with a few comments regarding our notation: For any function $\phi\in L^2(\Omega)$, we use $\phi_\e^k:=\phi_{|\Omega_\e^k}$ to denote its restriction to $\Omega_\e^k$ for $k=f,s,g$.
For the coefficients of the model, which are mostly assumed to be piecewise constant, we suppress the superscripts and the $\e$-dependency for better readability wherever possible.
For example, for the mass density, which is assumed to be constant in every subdomain $\Omega_\e^k$ ($k=f,s,g$), i.e., there are $\rho^f,\rho^g,\rho^s>0$ such that $\rho_\e(x)=\sum_{i=f,g,s}\chi_{\Omega_\e^k}(x)\rho^k$, we just write $\rho$. Here, $\chi : \Omega \to \{0, 1\}$ denotes the indicator function, e.g., $\chi_{\Omega_\e^k}(x) = 1$ only if $x \in \Omega_\e^k$. 
We also introduce the $\jump{\cdot}$-notation to denote the jump of a function across the subdomains (in the direction of the normal vectors introduced in the preceding section), e.g., in the above example of the mass density, we have $\jump{\rho}=\rho^f-\rho^g$ at $\Sigma_\e^f$. 

Now, let $\theta_\e$ denote the temperature.
For the fluid and solid domains, we consider parabolic heat equations
\begin{subequations}\label{eq:epsilon-problem}
\begin{linenomath*}\begin{alignat}{2}
    c\rho \partial_t \theta_\e - \dive{(\kappa \nabla \theta_\e - c \rho v_\varepsilon \theta_\e)} &= f_\varepsilon &&\text{ in } S \times \OmegaIdx{f},\label{epsilon-problem:1} \\
    c \rho \partial_t \theta_\e - \dive{(\kappa \nabla \theta_\e)} &= f_\varepsilon \quad &&\text{ in } S \times \OmegaIdx{s}.\label{epsilon-problem:2}
\end{alignat}\end{linenomath*}
Here, $v_\e \in L^2(S, W^{1,\infty}(\OmegaIdx{f}))^d$
is a given velocity with $\nabla \cdot v_\varepsilon = 0$ and $v_\varepsilon = 0$ on $\Sigmafs \cup \Sigmafg$.
In each subdomain, we have the specific heat $c$, mass density $\rho$, heat conductivity $\kappa$, and heat source $f_\varepsilon$.  
At the interface between fluid and solid, we consider perfect heat transfer
\begin{linenomath*}\begin{alignat}{2}
        \jump{\theta_\e} &= 0 \quad  &&\text{ on } S \times \Sigmafs, \label{eq:1c}
        \\
        \jump{\kappa \nabla \theta_\e} \cdot n _\e
        &= 
        0 \quad &&\text{ on } S \times \Sigmafs. \label{eq:1d}
\end{alignat}\end{linenomath*}
%
Inside the grain structures, a scaled heat equation is utilized, such that for $\varepsilon \to 0$, the contribution does not vanish
\begin{linenomath*}\begin{equation*}
    \frac{1}{\varepsilon}\rho c \partial_t \theta_\e -
    \dive{(\kappa_\varepsilon \nabla \theta_\e)} = \frac{1}{\varepsilon} f_\varepsilon \quad \text{ in } S \times \OmegaIdx{g}.
\end{equation*}\end{linenomath*}
On the interface between grains and the surrounding subdomains, we apply heat balance and thermal resistivity conditions
\begin{linenomath*}\begin{alignat}{2}
    \jump{\kappa \nabla \theta_\e} \cdot n_\varepsilon &= 0\quad &&\text{ on } S \times (\Sigmafg\cup\Sigmasg), \\
    \kappa \nabla \theta_\e \cdot n_\varepsilon &= \alpha\jump{\theta_\e} &&\text{ on } S \times \Sigmasg,\\
    \kappa \nabla \theta_\e \cdot n_\varepsilon &= \alpha\jump{\theta_\e} \quad&&\text{ on } S \times \Sigmafg.
\end{alignat}\end{linenomath*}
Here, $\alpha$ denotes the heat exchange coefficient, which can differ on each interface. For small values of $\alpha$ (in respect to $|\Sigma_\varepsilon|^{-1}$) the above equations represent near thermal isolation between the subdomains. On the other hand, larger values can approximate perfect heat transfer, like in Eqs \eqref{eq:1c} and \eqref{eq:1d}.
\begin{remark}\label{remark:thermal_res}
    One could also consider thermal resistivity conditions on the interface $\Sigmafs$, e.g., 
    \begin{linenomath*}\begin{alignat}{2}
    \jump{\kappa \nabla \theta_\e} \cdot n_\varepsilon &= 0\quad &&\text{ on } S \times \Sigmafs, \\
    \kappa \nabla \theta_\e \cdot n_\varepsilon &= \alpha\jump{\theta_\e} \quad&&\text{ on } S \times \Sigmafs
    \end{alignat}\end{linenomath*}
    instead of  Eqs \eqref{eq:1c} and \eqref{eq:1d}. With modifications to the extension operator, the presented homogenization procedure can also be applied to this case. See also the \cref{re:non_perfect_transfer_discon} for the corresponding homogenized transfer conditions.

    Applying perfect heat transfer on all interfaces would be a particular case of the results from \cite{Neuss07}. 
\end{remark}
The adaptation mentioned in \cref{remark:thermal_res} may describe the application of a grinding process more realistically, since thermal equilibrium is usually not expected because of the continuous supply of new coolant and the different thermal properties of fluid, grinding wheel composite, and grains.
Since the perfect heat transfer along $\Sigmafs$ has a simpler notation, we mainly work with this case. 

Finally, we pose homogeneous Neumann boundary conditions at the outer boundaries and initial conditions:
\begin{linenomath*}\begin{alignat}{2}
    -\kappa \nabla \theta_\e \cdot \nu &= 0 &&\text{ on } S \times \partial \Omega, \\
    \theta_\e &= \theta_{\varepsilon, 0} \quad &&\text{ in } \{0\} \times \Omega.
\end{alignat}\end{linenomath*}
\end{subequations}
Different boundary conditions could also be applied, for example Dirichlet conditions at the upper and lower boundaries of $\Omega$.
Similarly, periodic boundary conditions can be used since they fit naturally in the concept of two-scale convergence. More caution is required as soon as we change the boundary conditions for the grain structure $\OmegaIdx{g}$. This influences the homogenization procedure and may need further consideration since one has to take into account the convergence on the boundary sections $\partial \Omega \cap \partial \OmegaIdx{g}$.
\subsection{Assumptions on data}
For a function $\phi \in L^2(\OmegaIdx{k})$, $k=f,s$, denote with $\hat{\phi}$ the zero extension to $\Omega^k$.
To pass to the limit, we assume the following properties on the data:
\begin{enumerate}
    \myitem{(A1)}\label{item:A1} It is assumed that the problem parameters are constant in each subdomain and fulfill $\kappa^k, \rho^k, c^k >0$, for $k=f,g,s$ and $\alphafg, \alphasg > 0$. Also, we assume the following scaling for the heat conductivity: 
    \begin{itemize}
        \item[(a)]\label{case:disconnected} $\OmegaIdx{g}$ \textbf{disconnected}: $\kappa^g_\varepsilon = \varepsilon \kappa^g$
        \item[(b)]\label{case:connected} $\OmegaIdx{g}$ \textbf{connected}: $\kappa^g_\varepsilon = \frac{1}{\varepsilon} \kappa^g$ 
    \end{itemize}
    \myitem{(A2)}\label{item:A2} The initial condition $\theta_{\e,0}\in L^2(\Omega)$ satisfies 
    \begin{linenomath*}\begin{equation*}
        C_0 \coloneqq \sup_{\varepsilon>0}{\left(\|\theta_{\varepsilon, 0}\|_{L^2(\Omega_\e)} +
        \frac{1}{\sqrt{\varepsilon}}\|\theta_{\varepsilon,0}\|_{L^2(\OmegaIdx{g})}
        \right)} < \infty.
    \end{equation*}\end{linenomath*}
    \myitem{(A3)}\label{item:A3} The volume source $f_\e\in L^2(S\times\Omega)$ satisfies
    \begin{linenomath*}\begin{equation*}
        C_f \coloneqq \sup_{\varepsilon>0}{\left(\|f_\varepsilon\|_{L^2(S \times\Omega_\e)} +
        \frac{1}{\sqrt{\varepsilon}}\|f_\varepsilon\|_{L^2(S \times\OmegaIdx{g})}
        \right)} < \infty.
    \end{equation*}\end{linenomath*}
    \myitem{(A4)}\label{item:A4} For the velocity $v_\varepsilon\in L^\infty(S\times \OmegaIdx{f})$ it holds
    \begin{linenomath*}\begin{equation*}
        C_v \coloneqq \sup_{\varepsilon>0}{\|v_\varepsilon\|_{L^\infty(S \times \OmegaIdx{f})}} < \infty.
    \end{equation*}\end{linenomath*}
    \myitem{(A5)}\label{item:A5} There is a limit function $\theta_0 \in L^2(\Omega)$ such that 
    $ \theta_{\varepsilon, 0} \to \theta_0$ for $\varepsilon \to 0$
    in $L^2(\Omega)$.
    Additionally,
    \begin{enumerate}
        \item[(a)] there exists a $\theta^g_0 \in L^2(\Sigma \times \GrainCell)$, such that ${\theta_{\varepsilon, 0}}_{|\Omega_\e^g} \twosc \theta^g_0$ for $\varepsilon \to 0$.
        For the definition of the two-scale convergence (denoted by $\twosc$) in a thin layer, see \cref{def:two_scale} in \cref{sec:homogenization}.
        \item[(b)] there is a function $\theta^g_0 \in L^2(\Sigma)$, such that ${\theta_{\varepsilon, 0}}_{|\Omega_\e^g} \twosc \theta^g_0$ for $\varepsilon \to 0$.
    \end{enumerate}
    \myitem{(A6)}\label{item:A6} There is a limit $f \in L^2(S\times \Omega)$ such that $f_\varepsilon \to f$ in $L^2(S\times \Omega)$. 
    Additionally,
    \begin{enumerate}
        \item[(a)] a function $f^g\in L^2(S\times\Sigma\times \GrainCell)$ exists with ${f_\varepsilon}_{|\Omega_\e^g} \twosc f^g$ for $\varepsilon \to 0$.
        \item[(b)] there is a $f^g\in L^2(S\times\Sigma)$ such that ${f_\varepsilon}_{|\Omega_\e^g} \twosc f^g$ for $\varepsilon \to 0$.
    \end{enumerate}
    \myitem{(A7)}\label{item:A7} There is a function $v\in L^2(S;{H^1(\Omega^{f})})^d$ with $\nabla \cdot v = 0$, such that $\hat{v}_\e\to v \text{ in }\ L^2(S\times\Omega^{f})$ for $\varepsilon \to 0$.
\end{enumerate}
The scaling for the heat conductivity $\kappa_\varepsilon^g$ in \ref{item:A1}, for both cases, is often used in the literature; see \cite{Gahn17, Gahn18, Neuss07} for similar situations. This scaling aims to keep the influence of the diffusion in the limiting process. Different types of scaling may be considered in future work, similar to the studies in \cite{Gahn17}.
The remaining conditions \ref{item:A2}--\ref{item:A4} are needed to obtain solution bounds with specific $\varepsilon$ dependencies, cf. \cref{thm:existence}. The assumptions for the initial temperature $\theta^g_{\varepsilon, 0}$ and heat source $f^g_\varepsilon$ are, for example, fulfilled by constant functions, since the volume of $\OmegaIdx{g}$ scales with $\varepsilon$. The Assumptions \ref{item:A5}--\ref{item:A7} are needed to pass to the limit $\varepsilon \to 0$.
\begin{remark}
    Regarding the velocity limit $v$, generally the fluid movement would be modeled via the (Navier--)Stokes equation. It is well known that, at least for a small Reynolds number in regard to $\varepsilon$, an effective velocity exists and the velocity would be zero at $\Sigma$. Higher-order correctors could be used, which would lead to a slip velocity along $\Sigma$; see \cite[Section 2 and 3]{Achdou98}.
\end{remark}
In the following, the subscript $\Tilde{\#}$ indicates that a function space contains functions that are periodic in the directions $1,\dots, d-1$, for example
\begin{linenomath*}\begin{equation*}
    H^1_{\Tilde{\#}}(Y^d) \coloneqq \left\{ 
        u \in H^1_{\text{loc}}(\R^d) : u_{|Y^d} \in H^1(Y^d), u(y+e_i) = u(y) \text{ for almost all } y \in Y^d \text{ and } i=1,\dots, d-1
    \right\}.  
\end{equation*}\end{linenomath*}
\subsection{Auxiliary results}
Here, we collect two auxiliary lemmas regarding extension and trace operators needed to carry out the following analysis and homogenization. 
\begin{lemma}[Extension operator]\label{lem:extension_operators}
    There exists a family of linear extension operators ${\mathbb{E}_\e:H^1(\Omega_\e) \to H^1(\Omega)}$ such that
    \begin{linenomath*}\begin{equation*}
        \|\mathbb{E}_\e \phi\|_{H^1(\Omega)} \leq C_{ext} \|\phi\|_{H^1(\Omega_\e)} \quad \text{for all } \phi \in H^1(\Omega_\e),
    \end{equation*}\end{linenomath*}
    where $C_{ext}>0$ is independent of $\varepsilon$.
\end{lemma}
\begin{proof}
    By construction, $\Omega_\e$ has a Lipschitz boundary.
    Therefore, we can utilize available results for extension operators, see \cite{acerbi92} or \cite[Theorem 2.2]{EB14} 
    for the connected and \cite[Theorem 
    2.10]{Cioranescu01} for the disconnected case, and the statement follows.
\end{proof}
\begin{lemma}[Trace estimate]\label{lem:trace_estimate}
Let $\{x\in \partial \GrainCell: n_\Gamma(x)\cdot e_d = 0\}$ be a null set in dimension $d-1$. Then,
there is an $\e$-independent $C_{tr}$ such that, for all $\phi \in H^1(\Omega_\e)$, it holds
\begin{linenomath*}\begin{equation*}
    \|\phi\|_{L^2(\Sigma_\e^f)}+\|\phi\|_{L^2(\Sigma_\e^s)} \leq C_{tr} \| \phi\|_{H^1(\Omega_\e)}.
\end{equation*}\end{linenomath*}
\end{lemma}
\begin{proof}[Proof of \cref{lem:trace_estimate}]
    Under the given assumptions, the estimate follows by using the extension operator from \cref{lem:extension_operators} and a coordinate transform followed by a standard trace estimate.
    As there are some technical details in the proof, we present it here.
    We only present the arguments from the side of the subdomain $\OmegaIdx{f}$; the estimate on $\Sigmasg$ follows in the same way.
    
    Given the assumptions on $\partial \GrainCell$, we can find $Y^{d-1}$-periodic Lipschitz functions $\gamma_i:\omega_i\subseteq \overline{Y}^{d-1} \to [0, 1]$ for $i \in I$, with finite $I \subset \mathbb{N}$, such that
    \begin{linenomath*}\begin{equation}
    \label{eq:grap_rep_eps}
        \Sigmafg = \bigcup_{i \in I} \, \Sigma^f_{i, \e} \coloneqq \bigcup_{i \in I} \, \left\{
            \left(\tilde{x}, \varepsilon\gamma_i\left(\frac{\tilde{x}}{\e}\right)\right) : \tilde{x} \in \omega_{i,\e}
            \right\},
    \end{equation}\end{linenomath*}
    where $\omega_{i,\e}$ is given by 
    \begin{linenomath*}\begin{equation*}
        \omega_{i,\e} = \text{int}\left(\bigcup_{\Tilde{k}\in\Z^{d-1}} \varepsilon \left(\overline{\omega_i} + \Tilde{k}\right)\right) \cap \Tilde{\Omega}.
    \end{equation*}\end{linenomath*}
    Using the identity (\ref{eq:grap_rep_eps}), we can build upon the ideas used in \cite[Proposition 2]{Donato19} to show a trace estimate in our case.
    First, we utilize the extension operators $\mathbb{E}_\e$ of \cref{lem:extension_operators} where it holds that
    \begin{linenomath*}\begin{equation*}
        \|\mathbb{E}_\e \phi\|_{H^1(\Omega)} \leq C_{ext} \|\phi\|_{H^1(\Omega_\e)} \quad \text{and} \quad \|\phi\|_{L^2(\Sigmafg)} = \|\mathbb{E}_\e \phi\|_{L^2(\Sigmafg)}.
    \end{equation*}\end{linenomath*}
    Next, we split up the integral over $\Sigmafg$ into integrals over multiple sections, each given by a graph of a height function, 
    \begin{linenomath*}\begin{equation*}
        \|\mathbb{E}_\e \phi\|_{L^2(\Sigmafg)}^2 = \sum_{i \in I} \|\mathbb{E}_\e \phi\|_{L^2(\Sigma^f_{i, \e})}^2.
    \end{equation*}\end{linenomath*}
    On each $\Sigma^{f}_{i, \e}$ we can compute the integral by the parameterization given by $\gamma_i$, which leads to
    \begin{linenomath*}\begin{equation*}
        \|\mathbb{E}_\e \phi\|_{L^2(\Sigma^{f}_{i, \e})}^2 
        = 
        \int_{\omega_{i, \e}} \left(\mathbb{E}_\e \phi\right)^2 \left(\tilde{x}, \varepsilon \gamma_i\left(\frac{\tilde{x}}{\varepsilon}\right)\right)
        \, \sqrt{1 + |\nabla_{\tilde{y}} \gamma_i(\Tilde{y})|^2_{\Tilde{y}=\frac{\title{x}}{\varepsilon}}}\di{\Tilde{x}} 
        \leq 
        C_{\gamma,i} \int_{\omega_{i, \e}} \left(\mathbb{E}_\e \phi\right)^2 \left(\tilde{x}, \varepsilon \gamma_i\left(\frac{\tilde{x}}{\varepsilon}\right)\right)
        \di{\Tilde{x}}
    \end{equation*}\end{linenomath*}
    where $C_{\gamma,i}<\infty$ since $\gamma_i$ is Lipschitz continuous and independent of $\varepsilon$. To further estimate the right-hand side, we use that Sobolev functions are absolutely continuous on almost all lines. Applying this argument in the direction $x_d$ together with the triangle inequality, we obtain
     \begin{linenomath*}\begin{equation}
     \label{eq:split_integral}
        \left\|\mathbb{E}_\e \phi\left(\tilde{x}, \varepsilon \gamma_i\left(\frac{\tilde{x}}{\varepsilon}\right)\right)\right\|_{L^2(\omega_{i, \e})}
        \leq 
        \left\|\mathbb{E}_\e \phi\left(\tilde{x}, 0\right)\right\|_{L^2(\omega_{i, \e})} 
        +
        \left\| \int_0^{\varepsilon \gamma_i\left(\frac{\tilde{x}}{\varepsilon}\right)} | \nabla( \mathbb{E}_\e \phi) \left(\tilde{x}, x_d\right) \cdot e_d| \di{x_d} \right\|_{L^2(\omega_{i, \e})}
    \end{equation}\end{linenomath*}
    The first integral on the right-hand side of Eq (\ref{eq:split_integral}) can be bounded with a trace estimate on the domain $\Omega^f$, which is independent of $\varepsilon$, 
    \begin{linenomath*}\begin{equation*}
        \left\|\mathbb{E}_\e \phi\left(\tilde{x}, 0\right)\right\|_{L^2(\omega_{i, \e})} 
        \leq
        \left\|\mathbb{E}_\e \phi\right\|_{L^2(\partial \Omega^f)}   \leq C         \left\|\mathbb{E}_\e \phi\right\|_{H^1(\Omega^f)} 
        \leq C C_{ext} \|\phi\|_{H^1(\Omega_\e)}. 
    \end{equation*}\end{linenomath*}
    In the second integral of  Eq \eqref{eq:split_integral}, we can integrate always to the height $\varepsilon$ instead of $\varepsilon \gamma_i$, then apply the Hölder inequality, and lastly estimate the integral over the small layer by the integral over the whole domain to obtain a bound,
    \begin{linenomath*}\begin{align*}
        \left\| \int_0^{\varepsilon \gamma_i\left(\frac{\tilde{x}}{\varepsilon}\right)} | \nabla( \mathbb{E}_\e \phi) \left(\tilde{x}, x_d\right) \cdot e_d| \di{x_d} \right\|_{L^2(\omega_{i, \e})} 
        &\leq \left\| \int_0^{\varepsilon} | \nabla( \mathbb{E}_\e \phi) \left(\tilde{x}, x_d\right) \cdot e_d| \di{x_d} \right\|_{L^2(\omega_{i, \e})} 
        \\
        &\leq \sqrt{\varepsilon} \left\| \nabla( \mathbb{E}_\e \phi) \cdot e_d\right\|_{L^2(\omega_{i, \e} \times (0, \varepsilon))} 
        \\
        &\leq \sqrt{\e} \left\| \nabla( \mathbb{E}_\e \phi)\right\|_{L^2(\Omega^f)} 
        \leq \sqrt{\e} C_{ext} \|\phi\|_{H^1(\Omega_\e)}.
    \end{align*}\end{linenomath*}
    By bringing everything together, we get the desired estimate for the trace operator
    \begin{linenomath*}\begin{equation*}
        \|\phi\|_{L^2(\Sigmafg)}^2 
        = \|\mathbb{E}_\e \phi\|_{L^2(\Sigmafg)}^2 
        = \sum_{i \in I} \|\mathbb{E}_\e \phi\|_{L^2(\Sigma^{f}_{i, \e})}^2
        \leq \sum_{i \in I} C_{\gamma,i} \left(C + \sqrt{\varepsilon}\right)^2 C_{ext}^2  \|\phi\|_{H^1(\Omega_\e)}^2 \eqqcolon C_{tr}^2 \|\phi\|_{H^1(\Omega_\e)}^2. 
    \end{equation*}\end{linenomath*}
    Since $I$ is finite, we have $C_{tr}< \infty$, and $C_{tr}$ can be bounded independent of $\varepsilon$ for all $\e < \e_0$.
\end{proof}
The assumption that the vertical interface sections have measure zero is important 
for the proof of \cref{lem:trace_estimate}, since it allows us to transform the integral over $\Sigmafg$ into an integral along the flat surface $\Sigma$. If this assumption is not fulfilled, we can not represent a vertical boundary over a graph that is defined on a subsection of $\Sigma$. However, it should be possible to approximate the trace on a vertical section with a slightly tilted section that can be represented as a graph along $\Sigma$, albeit with further technical estimates.
\section{Analysis of the micro model}\label{sec:analysis}
To carry out the homogenization, we first show that our model is well-posed and derive solution estimates. First, we introduce the weak formulation of the system (\ref{eq:epsilon-problem}). To this end, we consider the solution space  
\begin{linenomath*}\begin{equation*}
    W_\varepsilon = \left\{ u\in L^2(S\times\Omega) \, : \, \partial_tu  \in L^2(S\times\Omega),\, u|_{\Omega_\e} \in L^2(S; H^1(\Omega_\e)), \,  u|_{\OmegaIdx{g}} \in L^2(S; H^1(\OmegaIdx{g}))
    \right\}.
\end{equation*}\end{linenomath*}
We call $\theta_\e\in W_\varepsilon$ a weak solution of the problem (\ref{eq:epsilon-problem}) if $\theta_\e(0, \cdot) = \theta_{\varepsilon, 0}$ almost everywhere in $\Omega$ and 
\begin{linenomath*}\begin{multline}\label{eq:weak_formulation}
        (\rho c \partial_t \theta_\e, \varphi)_{\Omega_\e} +
        \frac{1}{\varepsilon}(\rho c \partial_t \theta_\e, \varphi)_{\OmegaIdx{g}}
        +
        (\kappa\nabla \theta_\e, \nabla\varphi)_{\Omega_\e}
        - (\rho c v_\varepsilon \theta_\e, \nabla\varphi)_{\OmegaIdx{f}}
        \\
        + (\e^{2\gamma}\kappa\nabla\theta_\e, \nabla\varphi)_{\OmegaIdx{g}} + (\alpha \jump{\theta_\varepsilon}, \jump{\varphi})_{\Sigmafg} + (\alpha \jump{\theta_\varepsilon}, \jump{\varphi})_{\Sigmasg}
        = (f_\varepsilon, \varphi)_{\Omega_\e}+\frac{1}{\varepsilon}(f_\varepsilon, \varphi)_{\OmegaIdx{g}}
\end{multline}\end{linenomath*}
holds for all $\varphi\in W_\e$ and almost all $t\in S$.
Here, $\gamma=\frac{1}{2}$ in Case (a) and $\gamma=-\frac{1}{2}$ in Case (b).
\begin{theorem}[Existence and bounds]\label{thm:existence}
Let the Assumptions \ref{item:A1}--\ref{item:A4} be fulfilled.
There exists a unique weak solution $\theta_\varepsilon \in W_\varepsilon$ satisfying $\theta_\e(0, \cdot) = \theta_{\varepsilon, 0}$ a.e. and Eq \eqref{eq:weak_formulation}.
In addition, it holds
\begin{linenomath*}\begin{equation}\label{eq:solution_esitmate}
        \|\theta_\e\|_{L^\infty(S; L^2(\Omega_\e))} +
        \|\nabla \theta_\e\|_{L^2(S; L^2(\Omega_\e))}
        + \e^{-\frac{1}{2}} \|\theta_\e\|_{L^2(S\times \OmegaIdx{g})} 
        + \varepsilon^{\gamma} \|\nabla \theta_\e\|_{L^2(S\times \OmegaIdx{g})}
        + \|\jump{\theta_\varepsilon}\|_{L^2(S\times(\Sigmafg \cup \Sigmasg))} \leq C,
\end{equation}\end{linenomath*}
for a $C<\infty$ independent on $\varepsilon$, with $\gamma=\frac{1}{2}$ in Case (a) and $\gamma=-\frac{1}{2}$ in Case (b). For the time derivative it holds
\begin{linenomath*}\begin{equation}\label{eq:time_derivative_estimate}
    \|\partial_t \theta_\e\|_{L^2(S\times\Omega_\e)}
    +
    \frac{1}{\sqrt{\varepsilon}}\|\partial_t \theta_\e\|_{L^2(S\times \OmegaIdx{g})} \leq C.
\end{equation}\end{linenomath*}
\end{theorem}
\begin{proof}
    For each $\varepsilon > 0$, the problem (\ref{eq:weak_formulation}) is a standard linear heat equation with jump conditions, which is an example of a linear parabolic PDE.
    As a result, it has a unique weak solution under the given assumptions, see e.g., \cite[Proposition 2.3]{Showalter2014-pm}.
    The estimate (\ref{eq:solution_esitmate}) follows through an energy argument.
    Testing with $\varphi=\theta_\varepsilon$ and integrating over the time interval $(0, t)$ yields
    \begin{linenomath*}\begin{equation*}\label{eq:theorem_existence}
        \begin{split}
            \|\frac{\rho c}{2} \theta_\e(t) \|^2_{L^2(\Omega_\e)} + 
            \frac{1}{\varepsilon} \||\frac{\rho c}{2} \theta_\e(t)\|^2_{L^2(\OmegaIdx{g})} + 
            \|\kappa \nabla \theta_\e\|^2_{L^2((0, t)\times \Omega_\e)}
            +
            \e^{2\gamma} \|\kappa\nabla \theta_\e\|_{L^2((0, t)\times\OmegaIdx{g})}
            + \|\alpha \jump{\theta_\varepsilon}\|^2_{L^2((0, t)\times(\Sigmafg \cup \Sigmasg))}
            \\
            = (\rho c v_\varepsilon \theta_\e, \nabla \theta_\e)_{L^2((0, t)\times\OmegaIdx{f})} + 
            (f_\varepsilon, \theta_\e)_{L^2((0, t)\times \Omega_\e)}
            +
            \frac{1}{\varepsilon}(f_\varepsilon, \theta_\e)_{L^2((0, t)\times\OmegaIdx{g})}
            + \|\frac{\rho c}{2} \theta_{\varepsilon, 0}\|^2_{L^2(\Omega_\e)}
            + 
            \frac{1}{2\varepsilon}\|\rho c\theta_{\varepsilon, 0}\|^2_{L^2(\OmegaIdx{g})}.
        \end{split}
    \end{equation*}\end{linenomath*}
    Applying \ref{item:A4} to the convection term leads to
    \begin{linenomath*}\begin{equation*}
        (\rho c v_\varepsilon \theta_\e, \nabla \theta_\e)_{\OmegaIdx{f}} \leq C_v  \|\rho c\theta_\e \|_{L^2(\OmegaIdx{f})} \|\nabla \theta_\e \|_{L^2(\OmegaIdx{f})} \leq \frac{\kappa^f}{2} \|\nabla \theta_\e \|_{L^2(\OmegaIdx{f})}^2 + \frac{(\rho^f c^f C_v)^2}{2\kappa_f} \|\theta_\e \|^2_{L^2(\OmegaIdx{f})}.
    \end{equation*}\end{linenomath*}
    Now, using the above inequality and the Assumptions \ref{item:A1}--\ref{item:A3} in combination with Gronwall's lemma, we arrive at the estimate (\ref{eq:solution_esitmate}).
    The estimate (\ref{eq:time_derivative_estimate}) follows in a similar way by {formally} testing with $\varphi=\partial_t\theta_\varepsilon$; see also \cite[Lemma 3.1]{Neuss07}.
\end{proof}
The previous Theorem estimates the jump over the edges $\Sigmafg$ and $\Sigmasg$; the trace can also be bounded independent of $\varepsilon$.  
\begin{lemma}[Estimate on $\Sigmafg$ and $\Sigmasg$]\label{lem:interface_estimate}
Let the Assumptions \ref{item:A1}--\ref{item:A4} be satisfied. For the solution $\theta_\varepsilon \in W_\varepsilon$ of Eq  \eqref{eq:weak_formulation}, it holds on the interfaces $\Sigmafg$ and $\Sigmasg$ that
\begin{linenomath*}\begin{equation*}
    \|\thetaIdx{f}\|_{L^2(S \times \Sigmafg)} + \|\thetaIdx{g}\|_{L^2(S \times \Sigmafg)} + \|\thetaIdx{s}\|_{L^2(S \times \Sigmasg)} + \|\thetaIdx{g}\|_{L^2(S \times \Sigmasg)} \leq C,
\end{equation*}\end{linenomath*}
for a $C<\infty$ independent of $\varepsilon$.
\end{lemma}
\begin{proof}
    Again, we only present the arguments for the interface $\Sigmafg$; the estimate on $\Sigmasg$ follows in the same way. Since both $\Omega_\e$ and $\OmegaIdx{g}$ have Lipschitz boundaries, there exist linear bounded trace operators $\mathbb{T}_\varepsilon : H^1(\Omega_\e) \to L^2(\partial \Omega_\e)$ and $\mathbb{T}^g_\varepsilon : H^1(\OmegaIdx{g}) \to L^2(\partial \OmegaIdx{g})$. This implies, together with Eq  \eqref{eq:solution_esitmate},
    \begin{linenomath*}\begin{equation}\label{eq:trace_estimate}
        \|\mathbb{T}_\varepsilon\theta_\e\|_{L^2(S \times \Sigmafg)} \leq 
        C_{tr}\|\theta_\e\|_{L^2(S ;H^1(\Omega_\e))} \leq C
    \end{equation}\end{linenomath*}
    for almost all $t\in S$ and by \cref{lem:trace_estimate}, $C_{tr}$ can be chosen independently of $\varepsilon$. For the estimate of $\thetaIdx{g}$ on $\Sigmafg$, we use that
    \begin{linenomath*}\begin{align*}
        \|\mathbb{T}^g_\varepsilon\thetaIdx{g}\|_{L^2(S \times \Sigmafg)} &\leq \|\mathbb{T}^g_\varepsilon\thetaIdx{g}- \mathbb{T}_\varepsilon\theta_\e\|_{L^2(S \times \Sigmafg)} + \|\mathbb{T}_\varepsilon\theta_\e\|_{L^2(S \times \Sigmafg)} 
        = \|\jump{\theta_\varepsilon}\|_{L^2(S \times \Sigmafg)} + \|\mathbb{T}_\varepsilon\theta_\e\|_{L^2(S \times \Sigmafg)} 
    \end{align*}\end{linenomath*}
    Combining the two estimates (\ref{eq:solution_esitmate}) and (\ref{eq:trace_estimate}) gives that $\thetaIdx{g}$ is also bounded on $\Sigmafg$.
\end{proof}
\section{Two-scale limit and homogenization}\label{sec:homogenization}
For passing to the limit $\varepsilon \to 0$, we apply the concept of two-scale convergence \cite{Allaire92}.
In the domain $\OmegaIdx{g}$ we need to consider the generalized two-scale convergence for thin domains, which was first introduced in \cite[Definition 4.1]{Neuss07} and further developed in \cite{Bhattacharya2022, Gahn21, Gahn17}.
We state here the main definitions and results we need for the limiting procedure.
\newpage
\begin{definition}[Two-scale convergence on thin domains]\label{def:two_scale}\hspace{5pt}
\begin{itemize}
    \item[i)] A sequence $u_\varepsilon\in L^2(S\times\OmegaIdx{g})$ is said to weakly two-scale converge to a function $u \in L^2(S\times \Sigma \times \GrainCell)$ (notation $u_\varepsilon \twosc u$) if 
    \begin{linenomath*}\begin{equation}
        \lim_{\varepsilon\to 0} \frac{1}{\varepsilon}\int_S \int_{\OmegaIdx{g}} u_\varepsilon(t, x)\varphi\left(t, \Tilde{x}, \frac{x}{\varepsilon}\right) \di{x}\di{t} = \int_S \int_{\Sigma} \int_{\GrainCell} u(t,\Tilde{x},y)\varphi(t,\Tilde{x},y) \di{y}\di{\Tilde{x}}\di{t},
    \end{equation}\end{linenomath*}
    for all $\varphi \in C(\overline{S \times \Sigma}; C_{\Tilde{\#}}(\overline{\GrainCell}))$.
    \item[ii)] A sequence $u_\varepsilon\in L^2(S\times\partial \OmegaIdx{g})$ is said to weakly two-scale converge to a function $u \in L^2(S\times \Sigma \times \partial \GrainCell)$ if 
    \begin{linenomath*}\begin{equation}
        \lim_{\varepsilon\to 0} \int_S \int_{\partial \OmegaIdx{g}} u_\varepsilon(t, x)\varphi\left(t, \Tilde{x}, \frac{x}{\varepsilon}\right) \di{\sigma_x}\di{t} = \int_S \int_{\Sigma} \int_{\partial \GrainCell} u(t,\Tilde{x},y)\varphi(t,\Tilde{x},y) \di{\sigma_y}\di{\Tilde{x}}\di{t},
    \end{equation}\end{linenomath*}
    for all $\varphi \in C(\overline{S\times \Sigma};C_{\Tilde{\#}}(\partial \GrainCell))$.
\end{itemize}
\end{definition}
\begin{lemma}[Two-scale limits]\label{lem:ts_limits} \hspace{5pt}
    \begin{itemize}
        \item[i)] Let $u_\varepsilon \in L^2(S; H^1(\OmegaIdx{g}))$ with
        \begin{linenomath*}\begin{equation*}
            \frac{1}{\sqrt{\varepsilon}}\|u_\varepsilon\|_{L^2(S\times \OmegaIdx{g})} 
            + \sqrt{\varepsilon} \|\nabla u_\varepsilon\|_{L^2(S\times \OmegaIdx{g})} \leq C.
        \end{equation*}\end{linenomath*}
        Then, there exists a function $u \in L^2(S\times \Sigma; H^1_{\Tilde{\#}}(\GrainCell))$ and a subsequence of $u_\varepsilon$, still denoted with $u_\varepsilon$, such that
        \begin{linenomath*}\begin{alignat*}{1}
            u_\varepsilon &\twosc u, \\
            \varepsilon \nabla u_\varepsilon &\twosc \nabla_y u.
        \end{alignat*}\end{linenomath*}
        \item[ii)] Let $\OmegaIdx{g}$ be connected and $u_\varepsilon \in L^2(S; H^1(\OmegaIdx{g}))$ with
        \begin{linenomath*}\begin{equation*}
            \frac{1}{\sqrt{\varepsilon}}\|u_\varepsilon\|_{L^2(S\times \OmegaIdx{g})} 
            + \frac{1}{\sqrt{\varepsilon}}\|\nabla u_\varepsilon\|_{L^2(S\times \OmegaIdx{g})} \leq C.
        \end{equation*}\end{linenomath*}
        Then, there exist functions $u \in L^2(S; H^1(\Sigma))$ and $u_1\in L^2(S\times \Sigma; H^1_{\Tilde{\#}}(\GrainCell)/\R)$ such that, up to a subsequence of $u_\varepsilon$, one has
        \begin{linenomath*}\begin{alignat*}{1}
            u_\varepsilon &\twosc u, \\
            \nabla u_\varepsilon &\twosc \nabla_{\Tilde{x}}u + \nabla_y u_1.
        \end{alignat*}\end{linenomath*}
        \item[iii)] Let $u_\varepsilon \in L^2(S \times \partial\OmegaIdx{g})$ such that $\|u_\varepsilon\|_{L^2(S\times \partial\OmegaIdx{g})} 
        \leq C$. Then, there exist $u \in L^2(S\times \Sigma \times \partial \GrainCell)$, such that $u_\varepsilon \twosc u$, up to a subsequence.
        Here, $u$ is extended periodically with respect to $\Tilde{y}$.
    \end{itemize}
\end{lemma}
\begin{proof}
    The statement $(i)$ is found in \cite[Theorem 4.4 (i)]{Bhattacharya2022} and for $(ii)$ we refer to  \cite[Theorem 3.3]{Gahn17}.
    For the disconnected geometry, $(iii)$ can be found in \cite[Theorem 4.4 (ii)]{Bhattacharya2022} as an extension of earlier results from \cite[Proposition 4.2]{Neuss07}.
    This result transfers to the connected geometry noting that $|\partial\Omega_\e^g|=\mathcal{O}(1)$ in both cases via \cite[Lemma 4.3]{Bhattacharya2022}.
\end{proof}
With the above properties of two-scale convergence and the technical results from \cref{lem:extension_operators} and \ref{lem:trace_estimate}, we can now determine the effective model.
For a function $\vartheta\in\mathcal{W}_\e$, the gradient is only defined on $\Omega_\e$ and $\Omega_\e^g$ with a possible jump across their interface.
To avoid overflowing notation, we will use $\nabla\vartheta^g\in L^2(S\times\Omega_\e^g)^d$ to denote $\nabla \vartheta_{|\Omega_\e^g}$.
Moreover, we use $\widetilde{\nabla\vartheta}\in L^2(S\times\Omega)^d$ to denote the function $\nabla\vartheta\in L^2(S\times\Omega_\e)$ extended by zero to the whole of $\Omega$.
\subsection{Homogenization of Case (a)}
Based on the estimates of \cref{thm:existence} for the solution $\theta_\varepsilon$ we are able to identify the following limit behavior for $\varepsilon \to 0$.
\begin{lemma}\label{lem:limits_case_a}
    There are limit functions $\theta \in L^2(S;H^1(\Omega))$ and $\theta^g \in L^2(S\times \Sigma; H^1(\GrainCell))$ such that 
    \begin{linenomath*}\begin{alignat*}{3}
        i&)\ \chi_{\Omega_\e} \theta_\varepsilon \to \theta
         \ \ \text{in}\ L^2(S\times\Omega),\quad 
        &&ii&&)\  \widetilde{\nabla \theta_\varepsilon} \rightharpoonup\nabla \theta  \ \ \text{in}\ L^2(S\times\Omega)^d,
        \\
        iii&)\  \theta^g_\varepsilon\twosc \theta^g, 
        &&iv&&)\ \varepsilon \nabla \theta^g_\varepsilon
        \twosc \nabla_y \theta^g 
    \end{alignat*}\end{linenomath*}
    at least up to a subsequence. 
    Additionally, it holds 
    %
    \begin{linenomath*}\begin{equation}\label{eq:robin_convergence_a}
        \lim_{\varepsilon \to 0}\int_S \int_{\Sigma^{k}_\varepsilon} \jump{\theta_\e}\varphi\left(t, \Tilde{x}, \frac{x}{\varepsilon}\right) \di{\sigma_x} \di{t}
        =  \int_S \int_{\Sigma} \int_{\Gamma^{k}}(\theta - \theta^g) \varphi(t, \tilde{x}, y) \di{y} \di{\Tilde{x}} \di{t}
    \end{equation}\end{linenomath*}
    for $k=f,s$ and all admissible test functions $\varphi$.
\end{lemma}
\begin{proof}
    For the convergence of $\theta_\e$ in $\Omega_\e$, we utilize the extension operator from \cref{lem:extension_operators} and the estimates of \cref{thm:existence}.
    Since the function $\theta_\e$ is bounded in $W_\e$, the extension $\mathbb{E}_\e \left({\theta_\e}_{|\Omega_\e}\right)$ is also bounded and we obtain a weakly convergent subsequence in $L^2(S; H^1(\Omega))$.
    Note that, also, the time derivative of the extension exists and is bounded  \cite[Chapter 5]{hopker2016}.
    Since the embedding $H^1(\Omega) \hookrightarrow L^2(\Omega)$ is compact, we can apply Aubin-Lions Lemma \cite[Corollary 4]{Simon1986} to obtain strong convergence of $\mathbb{E}_\e \left({\theta_\e}_{|\Omega_\e}\right)$ in $L^2(S \times \Omega)$ for a subsequence.
    Since additionally $\chi_{\Omega_\e} \to 1$ in $L^2(\Omega)$, we obtain $i)$ and $ii)$.
    The points $iii)$ and $iv)$ follow from the estimate in \cref{thm:existence} and the two-scale convergence in \cref{lem:ts_limits} $i)$.

    For (\ref{eq:robin_convergence_a}) we demonstrate the arguments only for $k=f$, the solid part follows analogously. We start with the convergence of $\theta_\e$. Here, we utilize the extension operator $\mathbb{E}_\e$ and the assumption that the interface can be represented by multiple height functions. Therefore we obtain (if all the limits exist)
    \begin{align*}
        \lim_{\e \to 0} \int_{\Sigmafg} \theta_\e(x)\varphi\left(\Tilde{x}, \nicefrac{x}{\e}\right) \di{\sigma_x} 
        &= \lim_{\e \to 0} \sum_{i \in I} \int_{\omega_{i,\e}} \theta_\e\left(\Tilde{x}, \e \gamma_i\left(\nicefrac{\Tilde{x}}{\e}\right)\right) \varphi\left(\Tilde{x}, \nicefrac{\Tilde{x}}{\e}, \gamma_i\left(\nicefrac{\Tilde{x}}{\e}\right)\right)  \sqrt{1 + |\nabla_{\tilde{y}} \gamma_i(\Tilde{y})|^2_{\Tilde{y}=\frac{\title{x}}{\varepsilon}}}\di{\Tilde{x}}
        \\
        &= \underbrace{\lim_{\e \to 0} \sum_{i \in I} \int_{\omega_{i,\e}} \left[\theta_\e\left(\Tilde{x}, \e \gamma_i\left(\nicefrac{\Tilde{x}}{\e}\right)\right)
        - (\mathbb{E}_\e \theta_\e)(\tilde{x}, 0)\right] \varphi\left(\Tilde{x}, \nicefrac{\Tilde{x}}{\e}, \gamma_i\left(\nicefrac{\Tilde{x}}{\e}\right)\right)  \sqrt{1 + |\nabla_{\tilde{y}} \gamma_i(\Tilde{y})|^2_{\Tilde{y}=\frac{\title{x}}{\varepsilon}}}\di{\Tilde{x}}}_{(1)}
        \\
        &+ \underbrace{\lim_{\e \to 0} \sum_{i \in I} \int_{\omega_{i,\e}} \left[(\mathbb{E}_\e \theta_\e)(\tilde{x}, 0) - \theta(\tilde{x}, 0)\right] \varphi\left(\Tilde{x}, \nicefrac{\Tilde{x}}{\e}, \gamma_i\left(\nicefrac{\Tilde{x}}{\e}\right)\right)  \sqrt{1 + |\nabla_{\tilde{y}} \gamma_i(\Tilde{y})|^2_{\Tilde{y}=\frac{\title{x}}{\varepsilon}}}\di{\Tilde{x}}}_{(2)}
        \\
        &+ \underbrace{\lim_{\e \to 0} \sum_{i \in I} \int_{\omega_{i,\e}}  \theta(\tilde{x}, 0) \varphi\left(\Tilde{x}, \nicefrac{\Tilde{x}}{\e}, \gamma_i\left(\nicefrac{\Tilde{x}}{\e}\right)\right)  \sqrt{1 + |\nabla_{\tilde{y}} \gamma_i(\Tilde{y})|^2_{\Tilde{y}=\frac{\title{x}}{\varepsilon}}}\di{\Tilde{x}}}_{(3)}.
    \end{align*}
    Checking each of the terms individually we see:
    \begin{itemize}
        \item[(1)] On $\Sigmafg$ it holds $\theta_\e = \mathbb{E}_\e^f\thetaIdx{f}$ so obtain with the results of \cite[Lemma 3.2]{Donato10}
            \begin{align*}
                \left|(1)\right|
                &\leq C \sum_{i=1}^n \lim_{\e \to 0} \|\mathbb{E}_\e\theta_\e(\Tilde{x}, \e \gamma_i(\Tilde{x}/\e)) - \mathbb{E}_\e\theta_\e(\Tilde{x}, 0) \|_{L^2(\omega_{i, \e})}
                \underbrace{\|\varphi\left(\Tilde{x}, \tfrac{\Tilde{x}}{\e}, \gamma_i(\Tilde{x}/\e)\right)\|_{L^2(\omega_{i, \e})}}_{\leq C}
                \\
                &\leq C \sum_{i=1}^n \lim_{\e \to 0} \sqrt{\e} \|\mathbb{E}_\e \theta_\e \|_{H^1(\Omega)} \quad = 0.
            \end{align*}      
        \item[(2)] We can apply the same procedure as for (1) and then obtain the terms:
        \begin{equation*}
            \|\mathbb{E}_\e\theta_\e(\Tilde{x}, 0)  - \theta(\Tilde{x}, 0)\|_{L^2(\omega_{i, \e})} \leq  \|\mathbb{E}_\e\theta_\e(\Tilde{x}, 0)  - \theta(\Tilde{x}, 0)\|_{L^2(\Sigma)} \to 0, \quad \text{for } \e \to 0,
        \end{equation*}
        by the compactness of the trace operator in the bounded Lipschitz domain $\Omega^f$, which is independent of $\e$.
        \item[(3)] Since $\theta$ does not depend on $\e$ we obtain the desired limit 
        \begin{equation*}
            \lim_{\e \to 0} \int_{\Sigmafg} \theta(\Tilde{x}, 0) \varphi\left(\Tilde{x}, \tfrac{x}{\e}\right) \di{\sigma_x} = \int_{\Sigma} \int_{\Gammafg} \theta(\Tilde{x}, 0) \varphi\left(\Tilde{x}, y\right) \di{y} \di{\Tilde{x}}.
        \end{equation*}
    \end{itemize}
    
    The convergence of the grain temperatures on the interfaces $\Sigmafg$ and $\Sigmasg$ can be handled by a standard argument, see \cite[Section 5.3]{Neuss07} for a comparable setup. Nevertheless, we also quickly demonstrate the arguments for the present case for a better understanding. Using \cref{lem:interface_estimate} and the convergence  \cref{lem:ts_limits} $iii)$, one obtains that there exists a $u \in L^2(S\times \Sigma \times \partial \GrainCell)$ such that $\mathbb{T}^g_\e\theta^{g}_{\varepsilon} \twosc u$. By testing with any $\varphi \in C^1(\overline{\Sigma \times \GrainCell})^d $ that is periodic in $\Tilde{y}$ with period 1, setting $\varphi_\e = \varphi(\Tilde{x}, \frac{x}{\e})$ and using the following integration by parts
    \begin{linenomath*}\begin{align*}
        \int_{\Sigma}\int_{\GrainCell} \nabla_y \theta^g\cdot\varphi \di{y}\di{\Tilde{x}} &=
        \lim_{\varepsilon \to 0} \int_{\OmegaIdx{g}} \nabla \thetaIdx{g}\cdot \varphi_\e \di{x} 
        \\
        &= \lim_{\varepsilon \to 0} \left( - \int_{\OmegaIdx{g}} \thetaIdx{g} \dive_x \varphi_\e \di{x} - \frac{1}{\varepsilon}\int_{\OmegaIdx{g}} \thetaIdx{g} \dive_y \varphi_\e \di{x} + \int_{\partial \OmegaIdx{g}} \mathbb{T}^g_\e\theta^{g}_{\varepsilon} \varphi_\e \di{\sigma_x} \right) 
        \\
        &= - \int_{\Sigma}\int_{\GrainCell} \theta^g \dive_y \varphi \di{y}\di{\Tilde{x}} + \int_{\Sigma}\int_{\partial \GrainCell} u \varphi \di{\sigma_y}\di{\Tilde{x}}
        \\
        &= \int_{\Sigma}\int_{\GrainCell} \nabla_y \theta^g\cdot \varphi \di{y}\di{\Tilde{x}} - 
        \int_{\Sigma}\int_{\partial \GrainCell} \theta^g \varphi \di{\sigma_y}\di{\Tilde{x}}
        + \int_{\Sigma}\int_{\partial \GrainCell} u \varphi \di{\sigma_y}\di{\Tilde{x}}, 
    \end{align*}\end{linenomath*}
    we obtain that on $\partial Z$ it holds $\theta^g = u$ for almost all $t \in S$.
\end{proof}
With the previous Lemma, we can now pass to the limit in the weak formulation (\ref{eq:weak_formulation}). For the present case, the limiting procedure is standard and we only list the general steps one has to follow. For a detailed consideration of $\e \to 0$, we refer to \cite[Section 4.1]{Eden2022effective} with a comparable setup. The general procedure when determining the effective model consists of three steps:
\begin{enumerate}
    \item Construct smooth test functions compatible with the concept of two-scale convergence. Use these functions in the weak formulation to pass to the limit $\e \to 0$. In the present situation, fitting test functions would be $\varphi \in C^\infty(\overline{S \times \Omega})$ and $\varphi^g \in C^\infty(\overline{S\times \Sigma}; C^\infty_{\Tilde{\#}}(\overline{\GrainCell}))$.
    \item Trying to simplify the effective equations by decoupling the equations and isolating cell problems.
    \item Using a density argument to show that the limit also holds for test functions from more general spaces, here for example for functions $\varphi \in L^2(S; H^1(\Omega))$. 
\end{enumerate}
Carrying out the above procedure, under the Assumptions \ref{item:A5}--\ref{item:A7}, and collecting all the limits, we obtain that the effective function fulfills $(\theta(0, \cdot), \theta^g(0, \cdot)) = (\theta_0, \theta^g_0)$ and
\begin{linenomath*}\begin{equation}\label{eq:weak_formulation_case_A}
    \begin{split}
        & (\rho c \partial_t \theta, \varphi)_{\Omega} + (\kappa\nabla \theta, \nabla\varphi)_{\Omega}  + (\rho c v \nabla \theta , \varphi)_{\Omega^f} + 
        (\rho c \partial_t \theta, \varphi)_{\Sigma \times \GrainCell} + (\kappa \nabla_y \theta, \nabla_y\varphi)_{\Sigma \times \GrainCell}
        \\
        &+(\alpha (\theta^f - \theta^g), (\varphi^f - \varphi^g))_{\Sigma \times \Gammafg} + (\alpha (\theta^s - \theta^g), (\varphi^s - \varphi^g))_{\Sigma \times \Gammasg} 
        = (f, \varphi)_{\Omega} + (f, \varphi)_{\Sigma \times \GrainCell},
    \end{split}
\end{equation}\end{linenomath*}
for all $\varphi = (\varphi^f, \varphi^s, \varphi^g) \in L^2(S; H^1(\Omega^f)) \times L^2(S; H^1(\Omega^s)) \times L^2(S \times \Sigma; H^1_{\Tilde{\#}}(\GrainCell))$ with $\varphi^f_{|\Sigma} = \varphi^s_{|\Sigma}$. The homogenization result, together with the strong formulation of the effective problem, is summarized in the following \cref{thm:hom_case_a}. Additionally, we show that the solution of Eq  \eqref{eq:weak_formulation_case_A} is unique, and therefore the complete sequence $\theta_\e$ converges, in $L^2$ and two-scale sense, to the homogenized solution.
\begin{theorem}[Homogenization in the disconnected domain (Case (a))]\label{thm:hom_case_a}
Let the Assumptions \ref{item:A1}--\ref{item:A7} be satisfied in their (a)-variants. Then, $\theta_\varepsilon \to \theta$ in $L^2(S\times \Omega)$, and $\theta^g_\varepsilon \twosc \theta^g$ in $L^2(S\times\Sigma\times \GrainCell)$ for $\varepsilon \to 0$, where
\begin{linenomath*}\begin{equation*}
    \theta \in L^2(S;H^1(\Omega))\ \text{ and } \ \theta^g \in L^2(S\times\Sigma; H^1(\GrainCell))
\end{equation*}\end{linenomath*}
such that 
\begin{linenomath*}\begin{equation*}
    (\partial_t \theta, \partial_t \theta^g) \in L^2(S \times \Omega) \times L^2(S \times \Sigma \times \GrainCell).
\end{equation*}\end{linenomath*}
\begin{subequations}\label{eq:hom_system_a}
    The limit $(\theta, \theta^g)$ is characterized as the unique weak solution of
    \begin{linenomath*}\begin{alignat}{2}
         \rho c \partial_t \theta - \dive{(\kappa \nabla \theta - \rho c v\theta)} &= f \quad &&\text{ in } S \times \Omega^f
         \\
         \rho c \partial_t \theta - \dive{(\kappa \nabla \theta)} &= f \quad &&\text{ in } S \times \Omega^s,
         \\
         \jump{\kappa \nabla \theta} \cdot n &= \sum_{k=f,s} \alpha^k \int_{\Gamma^k} \theta-\theta^g \di{\sigma_y}\quad &&\text{ on } S \times \Sigma, \label{eq:5d}
    \end{alignat}\end{linenomath*}
    with the macroscopic outer boundary and initial conditions
    \begin{linenomath*}\begin{alignat}{2}
        \theta &= \theta_0 \quad &&\text{ in } \{0\} \times \Omega, 
         \\
        \kappa \nabla \theta \cdot \nu &= 0 &&\text{ on } S \times \partial \Omega.
    \end{alignat}\end{linenomath*}
    Additionally, the system is coupled with cell problems on the interface $\Sigma$
    \begin{linenomath*}\begin{alignat}{2}
         \rho c \partial_t \theta^g - \dive_y{(\kappa \nabla_y \theta^g)} &= f^g \quad &&\text{ in } S \times \Sigma \times \GrainCell, 
         \\
         \kappa^g \nabla_{y} \theta^g \cdot n_\Gamma &= \alphafg (\theta - \theta^g) \quad &&\text{ on } S \times \Sigma \times \Gammafg, 
         \\
         \kappa^g \nabla_{y} \theta^g \cdot n_\Gamma &= \alphasg (\theta - \theta^g) &&\text{ on } S \times \Sigma \times \Gammasg, 
         \\
         \theta^g &= \theta^g_0 &&\text{ in } \{0\} \times \Sigma \times \GrainCell.
    \end{alignat}\end{linenomath*}
\end{subequations}
\end{theorem}
\begin{proof}
    The homogenization procedure was explained in the step above. For uniqueness, consider that there are two different sets of solutions $(\theta_j, \theta^g_j)_{j=1,2}$.
    Their difference is denoted by $\overline{\theta}$.
    Utilizing the previous assumptions and applying an energy estimate for the difference leads the fluid temperature to
    \begin{linenomath*}\begin{align*}
        \ddt \|\overline{\theta}\|_{L^2(\Omega^f)}^2 + \|\nabla \overline{\theta}\|^2_{L^2(\Omega^f)}
        &+ |\Gammafg| \|\overline{\theta}\|_{L^2(\Sigma)} \\
        &\leq C\left(\|v\|_{L^\infty(S\times\Omega^f)}\|\nabla \overline{\theta}\|_{L^2(\Omega^f)}\|\overline{\theta}\|_{L^2(\Omega^f)} + \|\overline{\theta}\|_{L^2(\Sigma \times \GrainCell)} \|\overline{\theta}\|_{L^2(\Omega^f)}\right),
    \end{align*}\end{linenomath*}
    inside the solid domain to
    \begin{linenomath*}\begin{align*}
        \ddt \|\overline{\theta}\|_{L^2(\Omega^s)}^2 + \|\nabla \overline{\theta}\|^2_{L^2(\Omega^s)}
        &+ |\Gammasg| \|\overline{\theta}\|_{L^2(\Sigma)}
        \leq C \|\overline{\theta}\|_{L^2(\Sigma \times \GrainCell)} \|\overline{\theta}\|_{L^2(\Omega^s)}
    \end{align*}\end{linenomath*}
    and lastly for the cell problems to the estimate
    \begin{linenomath*}\begin{align*}
        \ddt \|\overline{\theta}\|_{L^2(\Sigma \times \GrainCell)}^2 + \|\nabla \overline{\theta}\|^2_{L^2(\Sigma \times \GrainCell)}
        &+ \|\overline{\theta}\|_{L^2(\Sigma \times \Gammafg)} + \|\overline{\theta}\|_{L^2(\Sigma \times \Gammasg)}
        \\
        &\leq C\left(\|\overline{\theta}\|_{L^2(\Sigma \times \GrainCell)} \|\overline{\theta}\|_{L^2(\Omega^f)} + \|\overline{\theta}\|_{L^2(\Sigma \times \GrainCell)} \|\overline{\theta}\|_{L^2(\Omega^s)}\right).
    \end{align*}\end{linenomath*}
    Applying both Young's and Gronwall's inequalities, we can conclude that $\overline{\theta} \equiv 0$ almost everywhere. Since the solution of  Eq \eqref{eq:hom_system_a} is unique, it follows that the whole sequence $\theta_\e$ converges. 
\end{proof}
\subsection{Homogenization of Case (b)}\label{sec:homogenization_b}
\begin{lemma}\label{lem:limits_case_b}
    There are limit functions $\theta \in L^2(S;H^1(\Omega))$, $\theta^g \in L^2(S; H^1(\Sigma))$, and $\theta^g_1 \in L^2(S\times \Sigma; H^1_{\Tilde{\#}}(\GrainCell)/\R)$ such that 
    \begin{linenomath*}\begin{alignat*}{3}
        i&)\ \chi_{\Omega_\e} \theta_\varepsilon \to \theta
         \ \ \text{in}\ L^2(S\times\Omega),\quad 
        &&ii&&)\  \widetilde{\nabla \theta_\varepsilon}\rightharpoonup\nabla \theta  \ \ \text{in}\ L^2(S\times\Omega)^d,
        \\
        iii&)\  \theta^g_\varepsilon\twosc \theta^g, 
        &&iv&&)\ \nabla \theta^g_\varepsilon
        \twosc \nabla_{\Tilde{x}} \theta^g + \nabla_y \theta^g_1 
    \end{alignat*}\end{linenomath*}
    at least up to a subsequence.
    Additionally, it holds 
    %
    \begin{linenomath*}\begin{equation*}\label{eq:robin_convergence_case_b}
        \lim_{\varepsilon \to 0}\int_S \int_{\Sigma^{k}_\varepsilon} \jump{\theta_\e}\varphi\left(t, \Tilde{x}, \frac{x}{\varepsilon}\right) \di{\sigma_x} \di{t}
        =  \int_S \int_{\Sigma} \int_{\Gamma^{k}}(\theta - \theta^g) \varphi(t, \tilde{x}, y) \di{y} \di{\Tilde{x}} \di{t}
    \end{equation*}\end{linenomath*}
    for $k=f,s$ and all admissible test functions $\varphi$.
\end{lemma}
\begin{proof}
    The proof follows similarly to \cref{lem:limits_case_a} by utilizing the estimates in \cref{thm:existence}, the boundedness of the extension operators, and the properties of two-scale convergence given in \cref{lem:ts_limits}.
\end{proof}
Utilizing \cref{lem:limits_case_b}, we are also able to pass to the limit for Case (b). Again, the limit procedure is standard, similar to the previous section and therefore skipped. The derivation of the effective conductivity $\Tilde{\kappa}$ is also standard (see for example \cite[Section 2]{Allaire92}) and results from the fact that the gradient of $\theta^g_1$ can be represented by a linear combination of the derivatives of $\theta^g$ and the cell solutions $\psi_i$ given by Eq \eqref{eq:cell_helper_problem}. For completeness, we state the weak formulation; the limit $(\theta, \theta^g)$ fulfills $(\theta(0, \cdot), \theta^g(0, \cdot)) = (\theta_0, \theta^g_0)$ and the equation
\begin{linenomath*}\begin{equation}\label{eq:weak_formulation_case_B}
    \begin{split}
        & (\rho c \partial_t \theta, \varphi)_{\Omega} + (\kappa\nabla \theta, \nabla\varphi)_{\Omega}  + (\rho c v \nabla \theta , \varphi)_{\Omega^f} + 
        (|Z|\rho c \partial_t \theta^g, \varphi^g)_{\Sigma} + (\Tilde{\kappa} \nabla_{\Tilde{x}} \theta^g, \nabla_{\Tilde{x}}\varphi^g)_{\Sigma} 
        \\
        &+(\alphafg |\Gammafg| (\theta^f - \theta^g), (\varphi^f - \varphi^g))_{\Sigma} + (\alphasg |\Gammasg|(\theta^s - \theta^g), (\varphi^s - \varphi^g))_{\Sigma} 
        = (f, \varphi)_{\Omega} + (|Z| f^g, \varphi^g)_{\Sigma},
    \end{split}
\end{equation}\end{linenomath*}
for all $\varphi = (\varphi^f, \varphi^s, \varphi^g) \in L^2(S; H^1(\Omega^f)) \times L^2(S; H^1(\Omega^s)) \times L^2(S; H^1(\Sigma))$ with $\varphi^f_{|\Sigma} = \varphi^s_{|\Sigma}$. Similar to the previous sections, the complete results are stated in \cref{thm:hom_case_b}.
\begin{theorem}[Homogenization in the connected domain (Case (b))]\label{thm:hom_case_b}Let the Assumptions \ref{item:A1}--\ref{item:A7} be satisfied in their (b)-variants. In the limit $\varepsilon \to 0$, it holds that $\theta_\varepsilon \to \theta$ in $L^2(S\times \Omega)$ and $\theta^g_\varepsilon \twosc \theta^g$ in $L^2(S\times\Sigma)$, where
\begin{linenomath*}\begin{equation*}
    \theta \in L^2(S;H^1(\Omega)) \text{ and } \theta^g \in L^2(S; H^1(\Sigma))
\end{equation*}\end{linenomath*}
such that
\begin{linenomath*}\begin{equation*}
    (\partial_t \theta, \partial_t \theta^g) \in L^2(S \times \Omega) \times L^2(S \times \Sigma).
\end{equation*}\end{linenomath*}
\begin{subequations}\label{eq:hom_system_b}
    The limit $(\theta, \theta^g)$ is characterized as the unique weak solution of
    \begin{linenomath*}\begin{alignat}{2}
         \rho c \partial_t \theta - \dive{(\kappa \nabla \theta - \rho c v\theta)} &= f \quad &&\text{ in } S \times \Omega^f,
         \\
         \rho c \partial_t \theta - \dive{(\kappa \nabla \theta)} &= f \quad &&\text{ in } S \times \Omega^s,
         \\
         \jump{\kappa \nabla \theta}\cdot n&= \tilde{\alpha} (\theta-\theta^g)\quad &&\text{ on } S \times \Sigma, \label{eq:6d}
    \end{alignat}\end{linenomath*}
    with outer boundary and initial conditions
    \begin{linenomath*}\begin{alignat}{2}
        \theta &= \theta_0\quad &&\text{ in } \{0\} \times \Omega, 
         \\
        \kappa \nabla \theta \cdot \nu &= 0 &&\text{ on } S \times \partial \Omega,
    \end{alignat}\end{linenomath*}
    coupled with an interface temperature
    \begin{linenomath*}\begin{alignat}{2}
         |\GrainCell|\rho c \partial_t \theta - \dive_{\Tilde{x}}{(\Tilde{\kappa} \nabla_{\Tilde{x}} \theta)} &= |\GrainCell| f^g + \tilde{\alpha} (\theta-\theta^g)\quad &&\text{ in } S \times \Sigma, 
         \\
         \theta^g &= \theta^g_0 &&\text{ on } \{0\} \times \Sigma, 
         \\
         \Tilde{\kappa}^g \nabla_{\Tilde{x}} \theta^g \cdot \nu &= 0 &&\text{ on } S \times \partial \Sigma.
    \end{alignat}\end{linenomath*}
    Here, we use the notation $\nabla_{\Tilde{x}}u=(\partial_{x_1}u, \dots, \partial_{x_{d-1}}u, 0)$.
    The effective heat exchange coefficient is given by $\tilde{\alpha}=\alphafg |\Gammafg|+\alphasg |\Gammasg|$ and the effective conductivity $\Tilde{\kappa} \in \R^{d \times d}$ is given (for $i,j=1,\dots,d-1$) by 
    \begin{linenomath*}\begin{equation}\label{eq:effective_diffusion}
        \Tilde{\kappa}_{ij} = \kappa^g \int_{\GrainCell} (\nabla_y \psi_i + e_i) \cdot e_j \di{y} 
        \quad \text{and} \quad
        \Tilde{\kappa}_{dj} = \Tilde{\kappa}_{id}=0,
    \end{equation}\end{linenomath*}
    where $e_i$ is the $i$--th unit vector and $\psi_i$ are the zero-average solutions, with a period of 1 in $\Tilde{y}$, of
    \begin{linenomath*}\begin{equation}\label{eq:cell_helper_problem}
        \begin{split}
            -\Delta_y \psi_i &= 0 \quad \hspace{17pt} \text{ in } \GrainCell, \\
            -\nabla_y \psi_i \cdot n_\Gamma &= e_i \cdot n \quad \text{ on } \left(\partial \GrainCell \setminus \partial Y^d \right),
        \end{split}
    \end{equation}\end{linenomath*}
    and $\psi_d \equiv 0$.
\end{subequations}
\end{theorem}
\begin{proof}
    The limiting procedure is shown in \cref{lem:limits_case_b}. The uniqueness follows with a similar argument as in \cref{thm:hom_case_a}. 
\end{proof}
\begin{remark}\label{re:non_perfect_transfer_discon}
    In the case of non--perfect heat transfer between the fluid and solid domain mentioned in \cref{remark:thermal_res}, the effective temperature field $\theta$ would split into two functions $(\theta^f, \theta^s)$ that belong to the space $ L^2(S;H^1(\Omega^f)) \times L^2(S;H^1(\Omega^s))$, and Eq \eqref{eq:5d} as well as Eq \eqref{eq:6d} would be replaced with 
    \begin{linenomath*}\begin{alignat*}{2}
        \kappa^f \nabla \theta^f \cdot n &= \alphafs |\Gammafs| \left(\theta^f-\theta^s\right) + \alphafg \int_{\Gammafg} \theta^f-\theta^g \di{\sigma_y} \quad &&\text{ on } S \times \Sigma,
        \\
        \kappa^s \nabla \theta^s \cdot n &= \alphafs |\Gammafs| \left(\theta^f-\theta^s\right) + \alphasg \int_{\Gammasg} \theta^s-\theta^g \di{\sigma_y} \quad &&\text{ on } S \times \Sigma,
    \end{alignat*}\end{linenomath*}
    so that both the macroscopic temperature and heat flow are discontinuous across $\Sigma$.
\end{remark}
%
\section{Numerical simulations}\label{sec:simulations}
In this section, we introduce a numerical approach for solving the effective models and verify our homogenization results through a comparison with direct numerical simulations of the microscale model (\ref{eq:epsilon-problem}). In the following, we consider a rectangular domain $\Omega = [0, L] \times [0, B]\times [-H, H]$. The simulation studies are carried out with the FEM library FEniCS \cite{Fenics}, and Gmsh \cite{gmsh} is utilized to generate the various meshes.

We use the following stationary Navier--Stokes equation to model the underlying flow field 
%
\begin{linenomath*}\begin{alignat*}{2}
    \rho (u_\e \cdot \nabla) u_\e &= \mu \Delta u_\e - \nabla p_\e
    \quad \quad &&\text{ in } \OmegaIdx{f}, 
    \\
    \nabla \cdot u_\e &= 0 \quad &&\text{ in } \OmegaIdx{f}, 
    \\
    u_\e &= 0 \quad &&\text{ on } \Sigmafg \cup \Sigmafs, 
    \\
    u_\e &= u_{in} \quad  &&\text{ on } \partial\OmegaIdx{f} \cap \{x_3 = H\}, 
    \\
    \mu \nabla u_\e \nu - p_\e \nu &= 0 \quad  &&\text{ on } \partial\OmegaIdx{f} \cap \{x_1 \in \{0, L\}\}, 
    \\
    u_\e &\text{ is $\e$-periodic in } x_2. 
\end{alignat*}\end{linenomath*}
In the homogenized case, the same equation is solved, just in $\Omega^f$ and on $\Sigma$ instead of $\OmegaIdx{f}$ and $\Sigmafg \cup \Sigmafs$. Classical Taylor--Hood elements \cite{TAYLOR197373} of second and first order for velocity and pressure, respectively, are used to compute the solution. 

In the original problem with the resolved grain structures, we encounter discontinuous temperatures. To compute the solution, we utilize the discontinuous Galerkin method \cite[Chapter 4]{Riviere2008} 
with piecewise linear functions. To stabilize the diffusion advection equation, we utilize the SUPG method \cite{BrooksS82}. Under consideration of the used fluid velocity, a Dirichlet condition is applied at the upper boundary 
\begin{linenomath*}\begin{equation*}
    \theta = 0 \quad  \text{ on } \partial\OmegaIdx{f} \cap \{x_3 = H\}, 
\end{equation*}\end{linenomath*}
e.g., cooling fluid is supplied from the top and a periodic boundary condition in the $x_2$--direction.
These boundary conditions are more realistic regarding the motivating grinding process than homogeneous Neumann conditions.
Please note that these modified conditions were not specifically considered in our analysis.
Still, the periodic boundary condition could be incorporated without problems.
This also holds for the Dirichlet condition, since it is only set at the top of the domain away from the grain layer.

Next, we list the numeric values we used for the occurring parameters. The size of $\Omega$ is $L=B=H=1$. For the conductivity, we use $\kappa^f =0.1, \kappa^s=1.0, \kappa^g = 2.0$ and assume a normalization of heat capacity and density, to be precise $\rho^k, c^k = 1$ for $k=f,s,g$. For the fluid, the viscosity is set to $\mu=1$ and the constant inflow $u_{in}=(0, -1, 0)$ at $\partial\OmegaIdx{f} \cap \{x_3 = H\}$. With $h, h_{\GrainCell}, h_\Sigma$, we denote the largest diameter of the mesh elements inside $\Omega, \GrainCell$, and $\Sigma$, respectively. If not stated otherwise, we use $h=h_{\GrainCell}=h_\Sigma=0.05$. We set $\alpha = \alphafg = \alphasg$ and study the influence of the different exchange values. The heat source is set to $f^f=f^s=0$ and either $f^g\equiv 1.0$ or
\begin{linenomath*}\begin{equation*}
    f^g(x) = 
        \begin{cases}
            1 \quad &\text{ if } \sqrt{(x_1 - \frac{1}{2})^2 + (x_2 - \frac{1}{2})^2} \leq 0.3, \\
            0 &\text{ otherwise}.
        \end{cases}
\end{equation*}\end{linenomath*}
The constant heat source is only used in the convergence study $\varepsilon\to 0$, since simulating the whole domain for a small $\varepsilon$ was not feasible on our workstation.

As grain structures, we consider the cells shown in \cref{tab:pore_structures}: for the disconnected case, a sphere, and for the connected case, spheres connected with planar cylinders. The effective conductivity $\Tilde{\kappa}^g$ has been computed with a mesh size $h_{\GrainCell}=0.02$. 

\begin{table}[H]
    \centering
    \caption{The grain structures used in the numerical simulations. } \begin{tabular}{lcccc}
    	 \hline
          & Parameters & $|\Gammafg|, |\Gammasg|$ & $|\GrainCell|$ & $\Tilde{\kappa}$ \\
         \hline
         &&& 
         \\[\dimexpr-\normalbaselineskip+2pt]
         \noindent\parbox[c]{20mm}{\includegraphics[width=20mm, height=20mm]{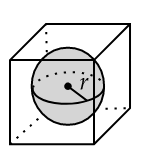}}
         & $r=0.4$ 
         & 1.01 & 0.27 & -- \\
        
         &&& 
         \\[\dimexpr-\normalbaselineskip+2pt]
         \noindent\parbox[c]{20mm}{\includegraphics[width=20mm, height=20mm]{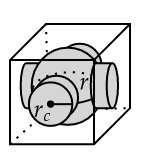}}
         & $r = 0.4, r_c = 0.2$
         & 1.12 & 0.34 & $\kappa^g \begin{pmatrix}
                    0.2 & 0.0 & 0.0 \\
                    0.0 & 0.2 & 0.0 \\
                    0.0 & 0.0 & 0.0 
                 \end{pmatrix}$
    \\\hline
    \end{tabular}
   
    \label{tab:pore_structures}
\end{table}

The homogenized model removes the $\e$-sized structures, which are expensive to numerically resolve from the problem, but introduces a new challenge; we numerically have to couple functions that are defined on the different domains $\Omega$, $\Sigma$ and $\Sigma \times \GrainCell$. Here, we realize this coupling by an iterative algorithm. At a given time step, we first fix the temperature field $\theta^g$ to calculate $\theta^f$ and $\theta^s$, then calculate $\theta^g$ with fixed $\theta^f$ and $\theta^s$, and then repeat this until we can detect convergence of the temperatures. This procedure is explained in more detail in \cref{algo:iterative_scheme}. Another possible approach would be to couple the functions directly by constructing one large linear system, similar to \cite{Gahn18_FEM} where a mixed FEM is analyzed. This, of course, would circumvent the need for iterations but would lead in Case (a) to a large system matrix. In the iterative scheme, we represent all temperature fields by continuous and piecewise linear functions. 
\begin{algorithm}[ht]
	\caption{Iterative scheme for one time step of the homogenized models}
    \label{algo:iterative_scheme}
	\KwIn{Information $\theta^{f}(t_n, \cdot), \theta^s(t_n, \cdot), \theta^g(t_n, \cdot)$, at time step $t_n$, and tolerance $\tau$.}
	Do the time step  $t_n \to t_{n+1}$, with $\theta^{f}(t_n, \cdot), \theta^s(t_n, \cdot), \theta^g(t_n, \cdot)$, to compute $\theta^{f}_{0}(t_{n+1}, \cdot), \theta^{s}_{0}(t_{n+1}, \cdot), \theta^{g}_{0}(t_{n+1}, \cdot)$.\\
    Set $i=0, E_i \geq \tau$. \\
	\While{$E_i \geq \tau$}{
    	Use $\theta^g_{i}(t_{n+1}, \cdot)$ to redo the time step $t_n \to t_{n+1}$ of $\theta^{f}$ and $\theta^{s}$, denote the solutions with $\theta^{f}_{i+1}(t_{n+1}, \cdot), \theta^{s}_{i+1}(t_{n+1}, \cdot)$. \\
        Redo the time step of $\theta^g$ with $\theta^{f}_{i+1}(t_{n+1}, \cdot), \theta^{s}_{i+1}(t_{n+1}, \cdot)$, denote the solution by $\theta^{g}_{i+1}(t_{n+1}, \cdot)$.\\
        Compute $E_{i+1}$ defined in Eq \eqref{eq:iterative_error}. \\
        Set i = i + 1.
        }
    \KwRet{$\theta^{f}_{i}(t_{n+1}, \cdot), \theta^{s}_{i}(t_{n+1}, \cdot) \text{ and } \theta^g_{i}(t_{n+1}, \cdot)$.}
\end{algorithm}

A similar iterative procedure was already applied in our previous work \cite{Eden2022effective}. There, we also showed, in case the time step is done with the backward Euler method, the convergence of the scheme \cite[Appendix A]{Eden2022effective}. The same argumentation can be applied in the present case. Denote with $\theta_{i}^k$ the solution in the $i$--th iteration of the algorithm and define the difference between the two following iterations
\begin{linenomath*}\begin{equation*}
    e_i^k = \|\theta^k_{i} - \theta^k_{i-1}\|_{L^2(\Omega^k)} \text{ and } e_i^g = \|\theta^g_{i} - \theta^g_{i-1}\|_{L^2(G)},
\end{equation*}\end{linenomath*}
for $i>0$ and $G=\Sigma \times \GrainCell$ or $G=\Sigma$ depending on Case (a) or Case (b). The total difference is given by 
\begin{linenomath*}\begin{equation}\label{eq:iterative_error}
    E_{i} =  \sqrt{(e_i^f)^2 + (e_i^s)^2 + (e_i^g)^2}.
\end{equation}\end{linenomath*}
At a given time point $t_n$ and iteration step $i\geq 2$ it then holds
\begin{linenomath*}\begin{equation}\label{eq:iterative_conv}
    E_{i} \leq C \left(\frac{\alphafg |\Gammafg| + \alphasg |\Gammasg|}{\alphafg |\Gammafg| + \alphasg |\Gammasg| + \frac{\rho^f c^f}{\Delta t} - (\rho^f c^f)^2 \|v\|_{L^\infty(S \times \Omega^f)}^2 \frac{1}{2\delta}}\right)^{i-1.5} \|\theta_{1}^f(t_n, \cdot) - \theta_{0}^f(t_n, \cdot)\|_{L^2(\Sigma)}, 
\end{equation}\end{linenomath*}
for $C > 0$ and $0 < \delta < 2\kappa^f$. If $v \in L^\infty(S\times \Omega^f)$ and for small enough time step $\Delta t$, we can ensure convergence of the scheme. In the case that 
\begin{linenomath*}\begin{equation*}
    \int_{\Omega^f} c^f \rho^f v \nabla \theta^f \theta^f = \frac{1}{2}\int_{\partial \Omega^f} c \rho \, v \, {\theta^f}^2 \geq 0,
\end{equation*}\end{linenomath*}
which holds for the considered boundary conditions in our simulations, one obtains
\begin{linenomath*}\begin{equation*}
    E_{i} \leq C \left(\frac{\alphafg |\Gammafg| + \alphasg |\Gammasg|}{\alphafg |\Gammafg| + \alphasg |\Gammasg| + \frac{\rho^f c^f}{\Delta t} + \Tilde{C}}\right)^{i-1.5} \|\theta_{1}^f(t_n, \cdot) - \theta_{0}^f(t_n, \cdot)\|_{L^2(\Sigma)}, 
\end{equation*}\end{linenomath*}
$\Tilde{C}>0$. In the above estimate, the convergence no longer requires a sufficiently small time step $\Delta t$. In all subsequent simulations, we choose the tolerance $\tau = 1.e-6$ in \cref{algo:iterative_scheme}.

To verify the accuracy of the numerical simulations, a convergence study with respect to the mesh resolution for the stationary problem is presented in the Appendix, see Figure \ref{fig:convergence_for_mesh_resolution}. For all models, both effective models and the model with resolved microstructure, we achieve for $\theta$ the expected convergence behavior of $\mathcal{O}(h^2)$ in the $L^2$--norm. 

\textbf{Investigation of the iterative algorithm.}
Before verifying the homogenized model in the limit $\varepsilon \to 0$, we first investigate the iterative algorithm for the homogenized model in more detail.
In Figure \ref{fig:iterative_convergence}, the convergence of the stationary temperature model for different $\alpha$ is shown. For increasing values of $\alpha$ (or surfaces $|\Gammafg|$ and $|\Gammasg|$) the convergence speed decreases noticeably, in both Cases (a) and (b). This trend can also be seen in the estimation (\ref{eq:iterative_conv}), where the term in the brackets also approaches 1 for increasing $\alpha$. Two possible ways to circumvent this are solving a mixed FEM directly, like mentioned above, or applying a scheme to speed up the convergence of the iterative procedure \cite{Polyak1964}. Here, a simple relaxation method of the type
\begin{linenomath*}\begin{equation*}
    \theta_{i+1} = \theta_{i} + \eta (\Tilde{\theta}_{i+1} - \theta_{i}),
\end{equation*}\end{linenomath*}
\begin{figure}[H]
	\centering
	\includegraphics[width=\linewidth]{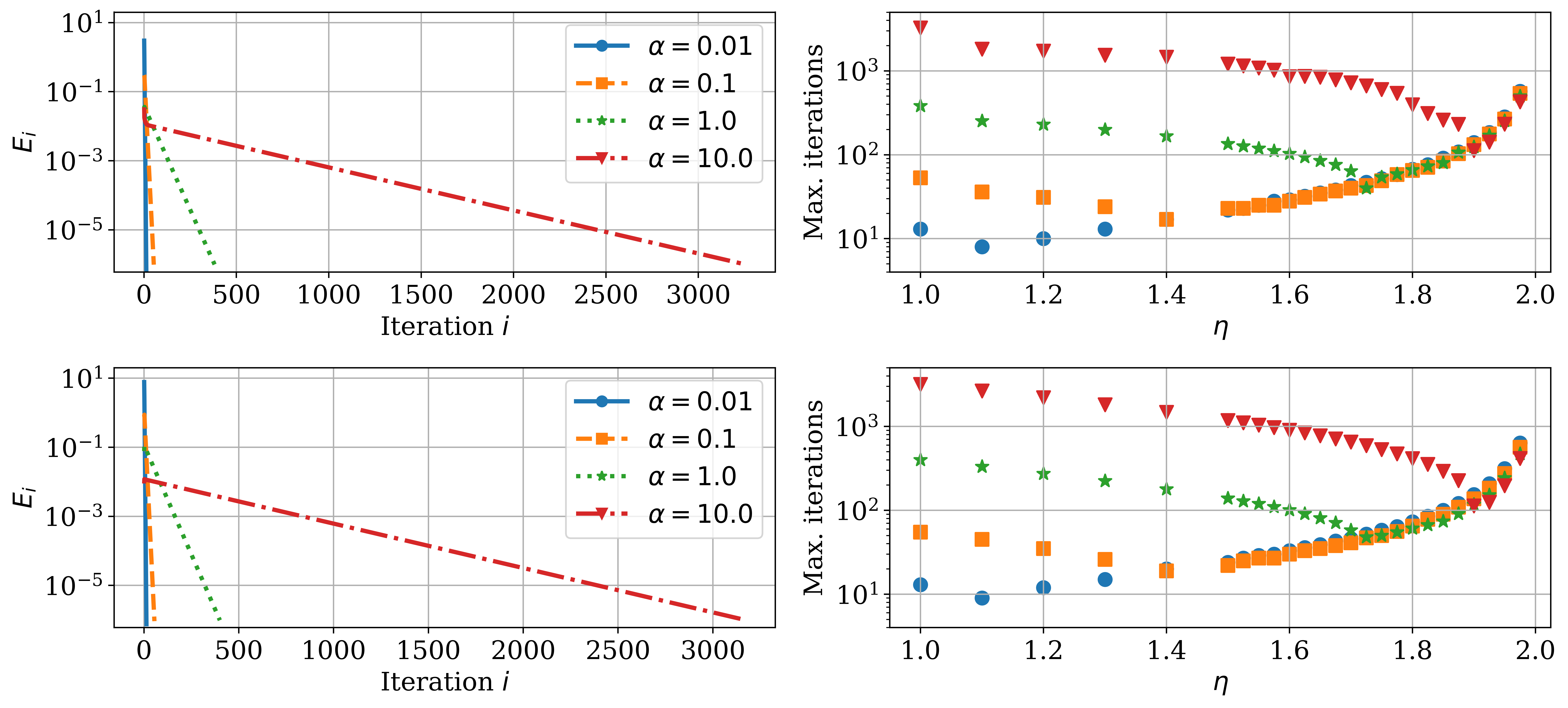}
	\caption{Convergence of the iterative \cref{algo:iterative_scheme} for different heat exchange values $\alpha=\alphafg=\alphasg$. At the top is the disconnected model and at the bottom is the connected case. The left side shows the linear trend of the difference between success iterations $E_i$. Depicted on the right is the maximal number of iterations with respect to $\eta$, until the tolerance $\tau$ is reached.}
	\label{fig:iterative_convergence}
\end{figure}

with $\eta \in [1, 2.0)$ and $\Tilde{\theta}_{i+1}$ the solution with respect to $\theta_i$, is used. The results are also presented in Figure \ref{fig:iterative_convergence}. We observe that this approach lowers the number of needed iterations considerably. Additionally, the optimal relaxation parameter $\eta$ increases with the heat exchange $\alpha$ (and/or surfaces $|\Gammafg|$ and $|\Gammasg|$). 
Last, we want to mention that we started in the stationary problem with $\theta_0 \equiv 0$. In the non-stationary case, one would start with the temperature of the previous time step, which generally should be close to the solution at the next step and therefore need fewer iterations. 

One additional aspect arises in the homogenized model  (\ref{eq:hom_system_a}) of the disconnected case. Here, the cell problems have to be computed on the whole domain $\Sigma \times \GrainCell$, which is of dimension higher than 3 and not directly implementable in FEniCS. In addition to the previously explained iterative procedure, we therefore also choose a discrete number of points $\{\mathbf{x}_j\}_{j=1}^M \subset \Sigma$ to first compute the cell temperature $\theta^g$ only at these specific points. To then evaluate $\theta^g$ for arbitrary $\Tilde{x} \in \Sigma$, we utilize an interpolation scheme in between the discrete points $\{\mathbf{x}_j\}_{j=1}^M$. To be consistent with the discretization of $\Sigma$, for each degree of freedom on $\Sigma$ one could solve the corresponding cell problem and then use an interpolation fitting to the discrete temperature space, e.g., linear interpolation for piecewise linear functions. Note that the use of such an interpolation introduces an additional consistency error in the numerical scheme.

In Figure \ref{fig:cell_number_influence} the influence of the position and number of $\{\mathbf{x}_j\}_{j=1}^M$ on the solution is demonstrated for the case $h=\frac{1}{16}$. For a given $M$, the points are distributed on a uniform grid in $[0, 1] \times [0,1]$. The grain temperature $\theta^g(\Tilde{x}, y)$ is constructed by linearly interpolating between the discrete cell problems $\theta^g(\mathbf{x}_j, y)$ with respect to $\tilde{x}$. We observe that, to obtain reasonable results, there must be enough cell problems to correctly resolve the circular heat source $f^g$. Interestingly, the solution still changes when more cells than the number of mesh vertices are used; see the results for $M>\frac{1}{h^2}$. This arises from the quadrature scheme in FEniCS, where points in between the degrees of freedom are used. 

We can conclude that the cell problems should be positioned such that all important areas and local effects are captured. One has to keep in mind that the computational effort grows with the number $M$. One advantage of the iterative scheme is that all cell problems can be solved independently, which allows for an easy parallelization of the problem. In the simulations of the previous and following Sections, we used $M=16^2$.
\begin{figure}[ht]
    \centering
    \begin{subfigure}[b]{0.43\textwidth}
         \centering
         \includegraphics[width=\textwidth]{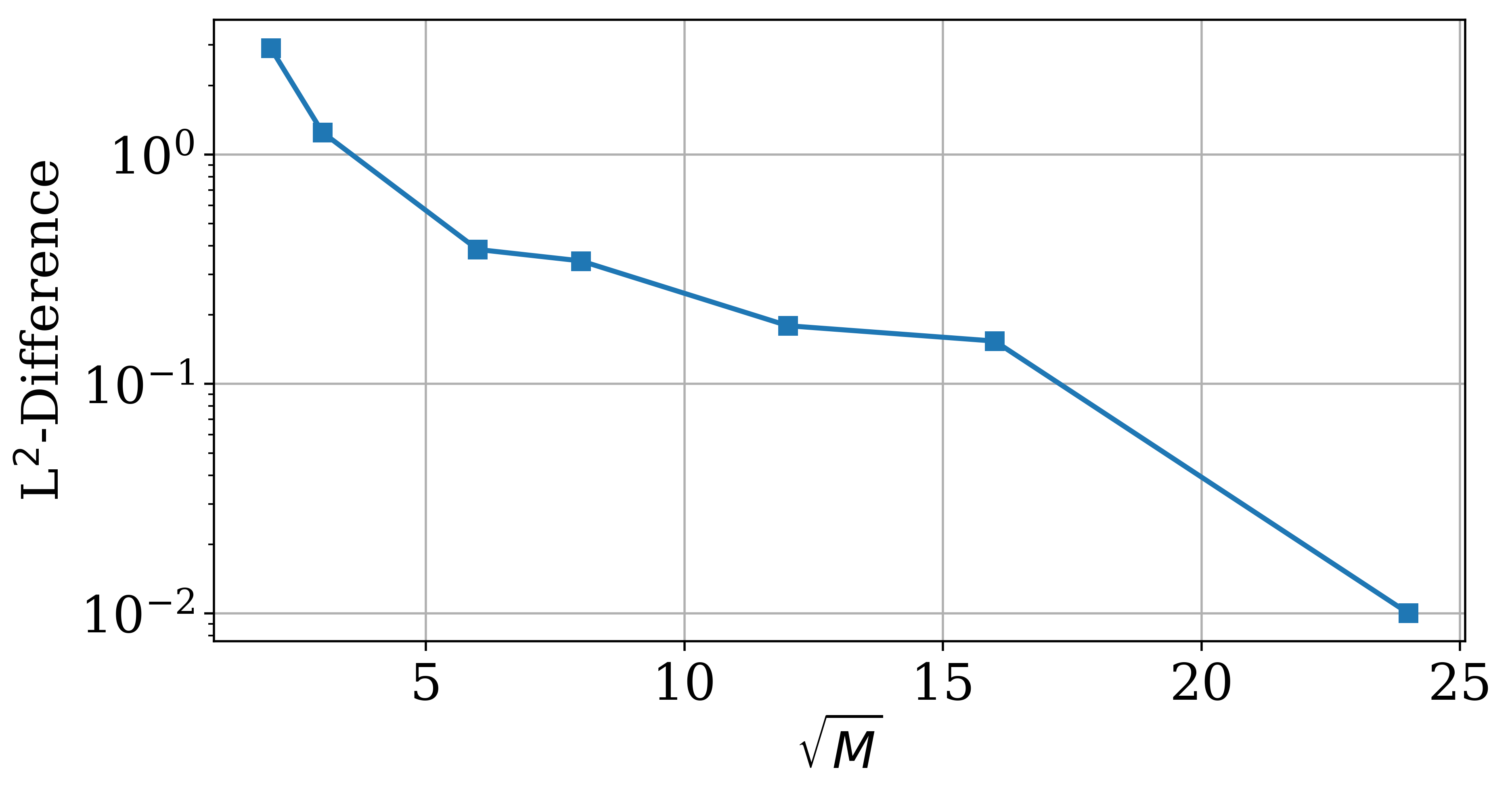}
         \vspace{15pt}
    \end{subfigure}
    \hfill
    \begin{subfigure}[b]{0.42\textwidth}
        \centering
        \begin{subfigure}[b]{0.3\textwidth}
             \centering
             \includegraphics[width=\textwidth]{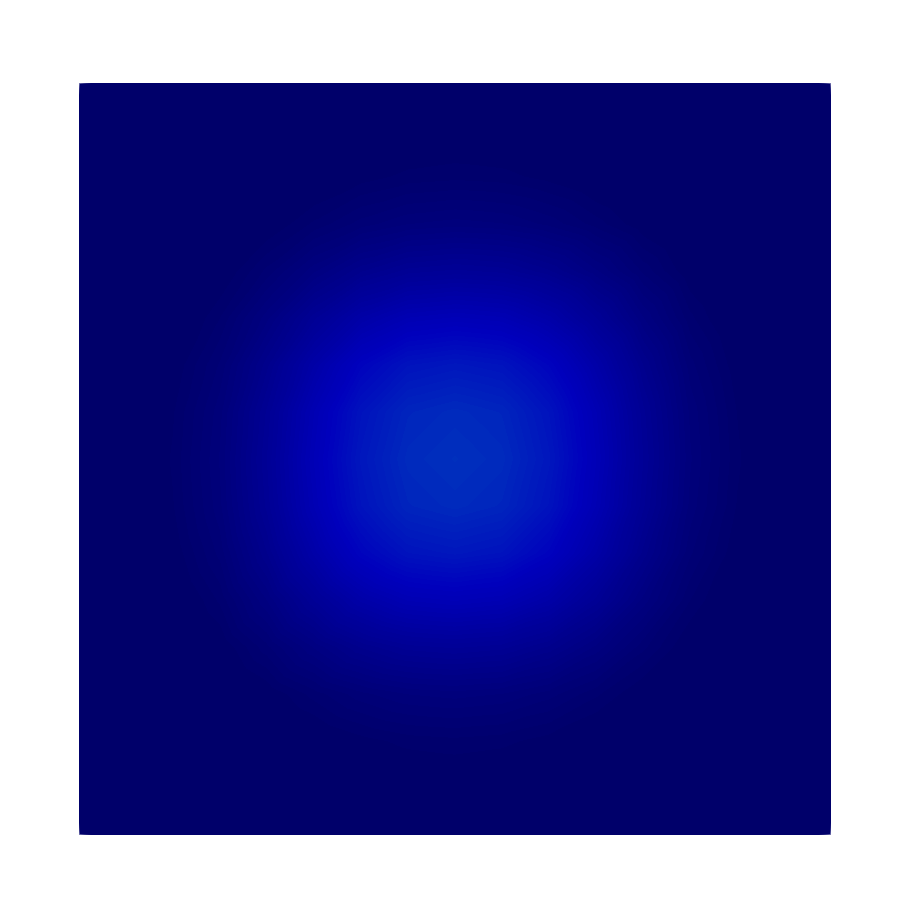}
             \caption{$M=3^2$}
        \end{subfigure}
        \begin{subfigure}[b]{0.3\textwidth}
             \centering
             \includegraphics[width=\textwidth]{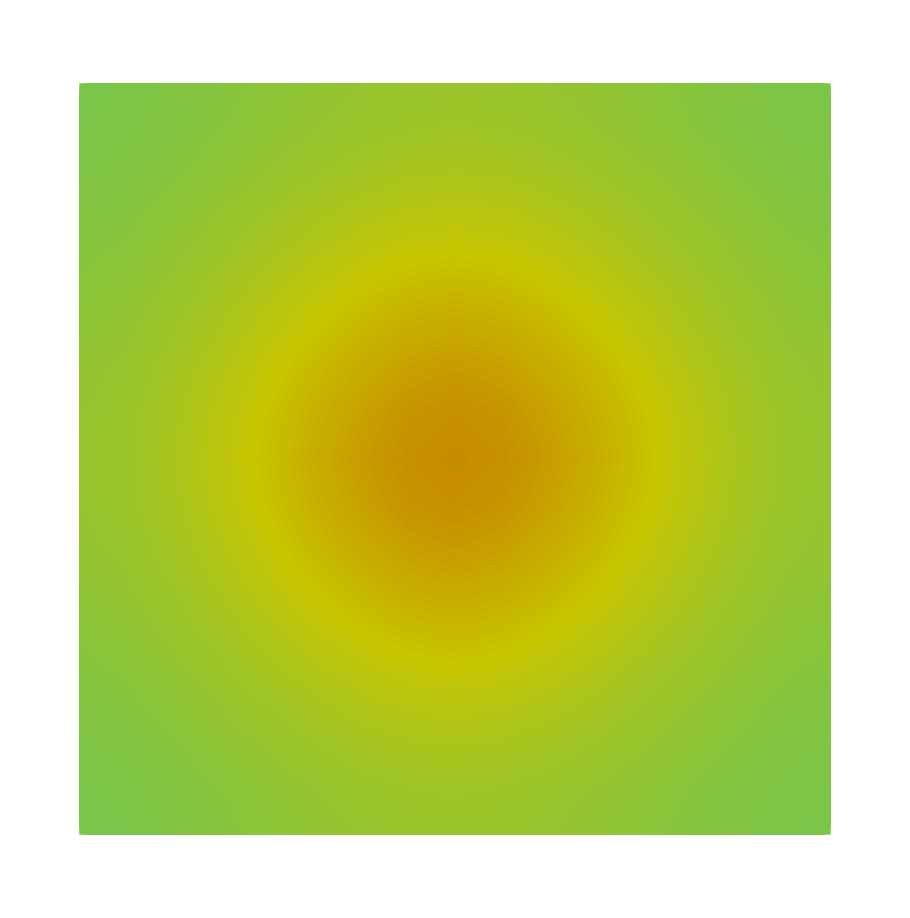}
             \caption{$M=6^2$}
        \end{subfigure}
        \begin{subfigure}[b]{0.3\textwidth}
             \centering
             \includegraphics[width=\textwidth]{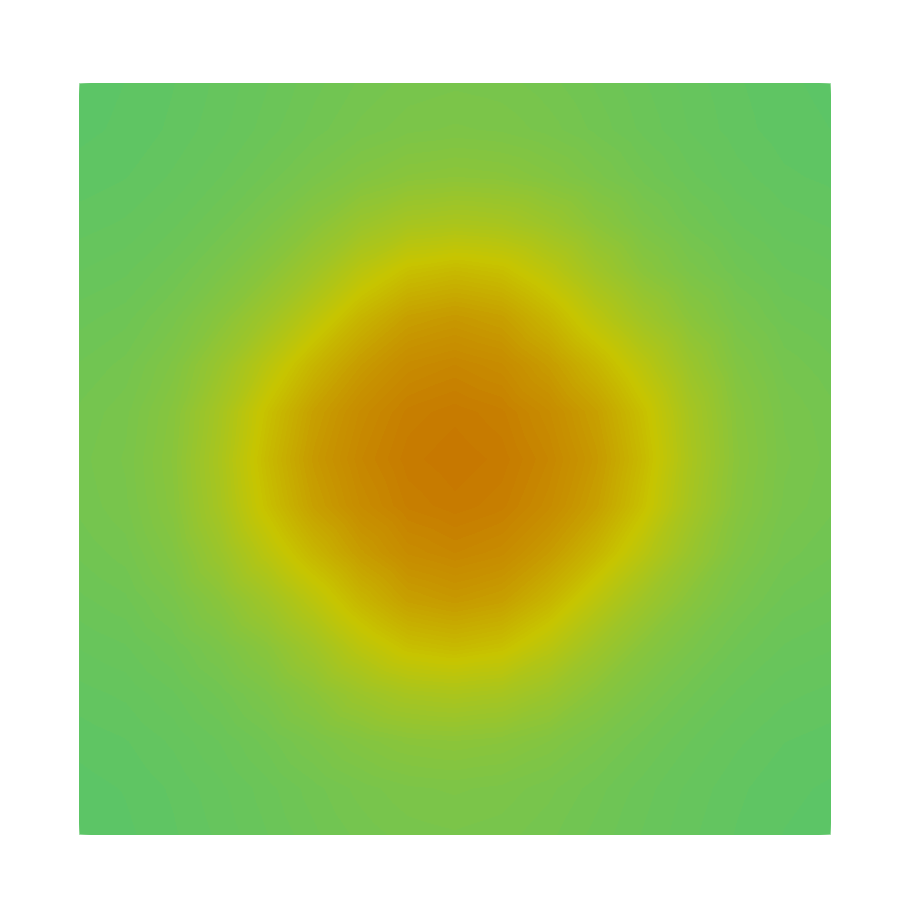}
             \caption{$M=8^2$}
        \end{subfigure}
        \\
        \begin{subfigure}[b]{0.3\textwidth}
             \centering
             \includegraphics[width=\textwidth]{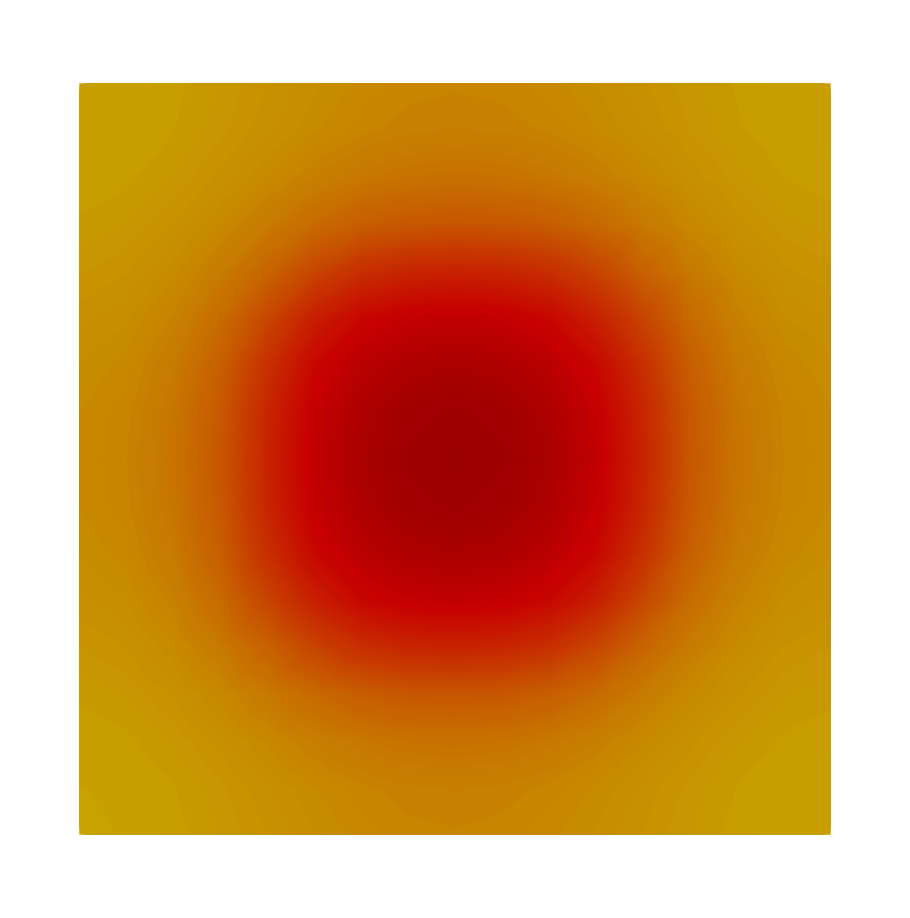}
             \caption{$M=16^2$}
        \end{subfigure}
        \begin{subfigure}[b]{0.3\textwidth}
             \centering
             \includegraphics[width=\textwidth]{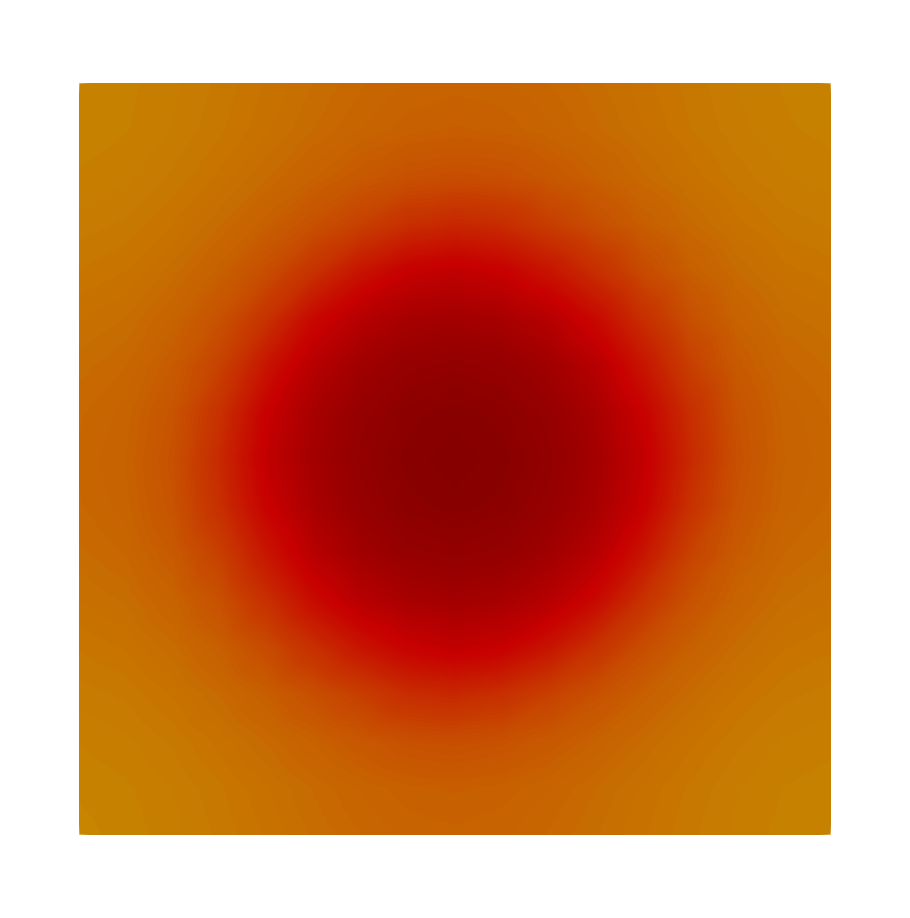}
             \caption{$M=24^2$}
        \end{subfigure}
        \begin{subfigure}[b]{0.3\textwidth}
             \centering
             \includegraphics[width=\textwidth]{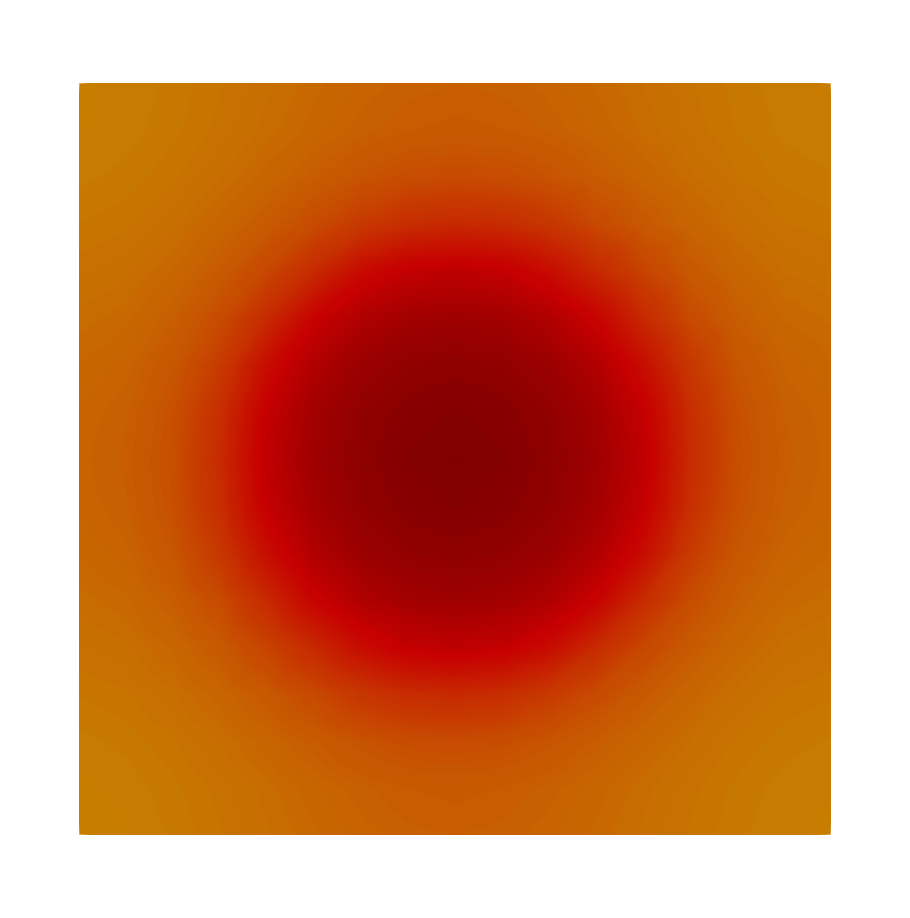}
             \caption{$M=32^2$}
        \end{subfigure}
    \end{subfigure} 
    \hfill
    \begin{subfigure}[b]{0.05\textwidth}
         \hspace{-35pt}
         \includegraphics[width=1.4\textwidth]{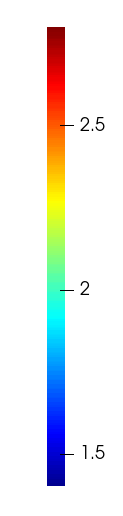}
         \vspace{18pt}
    \end{subfigure}
    \caption{Dependence of the numerical algorithm for model (\ref{eq:hom_system_a}) on the discrete position of the cell problems. On the left is the absolute $L^2$-difference of the temperature regarding the number of cell problems $M$, compared to the case $M=32^2$. On the right, the solution $\theta^f$ is depicted on $\Sigma$ for different $M$.}
    \label{fig:cell_number_influence}
\end{figure}

\textbf{Convergence study for the limit $\varepsilon \to 0$.}
To verify our homogenized temperature models (\ref{eq:hom_system_a}) and (\ref{eq:hom_system_b}), we carry out multiple simulations of the resolved microscale model for different $\varepsilon$ and compare the computed temperature with the effective models. We consider two different setups:
\begin{itemize}
    \item A simulation of Cases (a) and (b) in the previously described three-dimensional domain. Here, for small $\e$, the computational effort, while utilizing a fine-mesh resolution with respect to $\varepsilon$, became too large and could not be handled by our hardware. Therefore, we had to simplify the problem in terms of different aspects. First, we use a constant heat source $f^g=1$, and second, we remove the convection term. Neglecting the convection is acceptable, as we are mainly interested in the comparison of our derived effective temperature model with the resolved micro model, and the convection does not play a dominant role in the homogenization carried out in the previous sections. The effective behavior of flow over rough surfaces has also already been investigated in other studies, cf. \cite{Achdou98}. These modifications lead to a temperature profile that is $\e$--periodic in $x_1$ and $x_2$ and only varies in the vertical direction $x_3$. For simulating the model with resolved microscale we use the smaller domain  $\Omega = [0, 3\varepsilon] \times [0, 3\varepsilon] \times [-1, 1]$. Even in this scaled-down domain, more than 1 million simplices are required to resolve the microstructures. For this reason, we could not carry out more complex simulations for this setup. 
    \item A two-dimensional setup that is less expensive to solve numerically. Therefore, we can resolve the microstructure over the whole domain and also include the convection term. One disadvantage of the two-dimensional setup is that we can only investigate Case (a). Case (b), with connected grains as well as fluid and solid domains in contact, is not possible.
    Here, we set $\Omega=[0, 1]\times[-1, 1]$ and examine circular grains $\GrainCell$ with radius $r=0.4\e$. We keep the heat source $f^g=1$ and the Dirichlet boundary condition at the top boundary. For the fluid flow, we use the inflow $u_{in}=(0, -1)$ and free outflow at $x_1=0$ and $x_1=1$. 
\end{itemize}

\begin{figure}[H]
	\centering
	\includegraphics[width=0.8\linewidth]{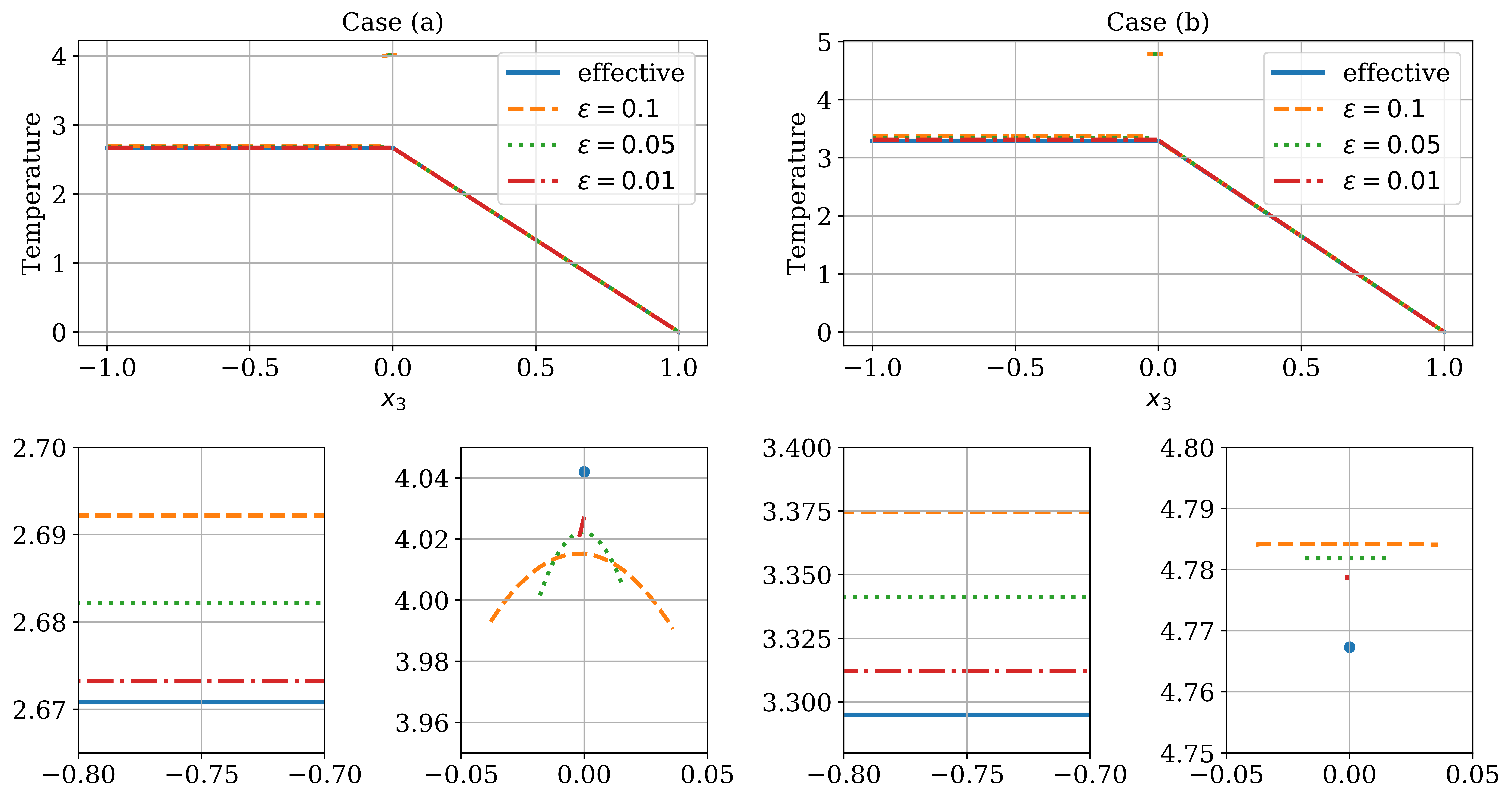}
	\caption{The fully resolved and effective temperature profile for both cases is depicted over the vertical line at $x_1=x_2=3\e/2$. On the left, Case (a), and on the right, Case (b). The top row shows the macroscopic temperature profile. At the bottom, two zoomed-in sections inside the solid domain $\Omega^s$ and around the interface grains show the difference between the solutions in more detail.}
	\label{fig:temperature_profile}
\end{figure}

The comparison for the three-dimensional case is shown in Figure  \ref{fig:temperature_profile}. The profile is plotted over a vertical line that goes through the center of the micro grains. Note that for Case (a), we cannot plot $\theta^g$ directly, since it also depends on $y$. In the zoomed-in segment, we therefore show the averaged temperature $\frac{1}{|\GrainCell|}\int_{\GrainCell}\theta^g$.
We observe that the difference between the effective model and the resolved micro model is small. In the resolved microscale model, one obtains slightly higher temperatures inside $\Omega^s$. Additionally, the effective model is able to capture the temperature values inside the grain domains, as shown in the zoomed-in plots. A convergence trend for $\varepsilon \to 0$ to the homogenized model also appears to be visible. But for the comparison, one has to also keep in mind the 
achieved numerical accuracy shown in Figure  \ref{fig:convergence_for_mesh_resolution} in the Appendix.

\begin{figure}[H]
	\centering
	\begin{subfigure}[b]{0.25\textwidth}
		\centering
		\includegraphics[width=\textwidth]{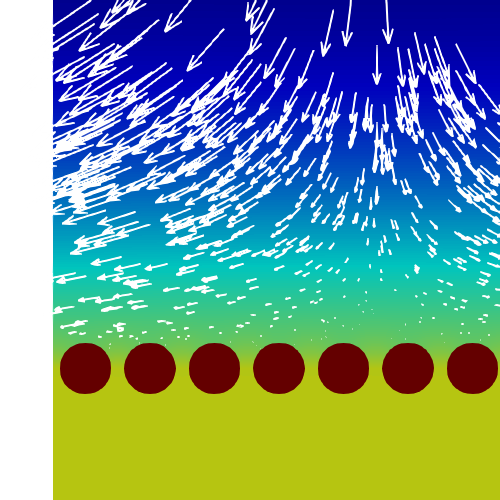}
		\caption{$\e=0.1$}
	\end{subfigure}
	\hfill
	\begin{subfigure}[b]{0.25\textwidth}
		\centering
		\includegraphics[width=\textwidth]{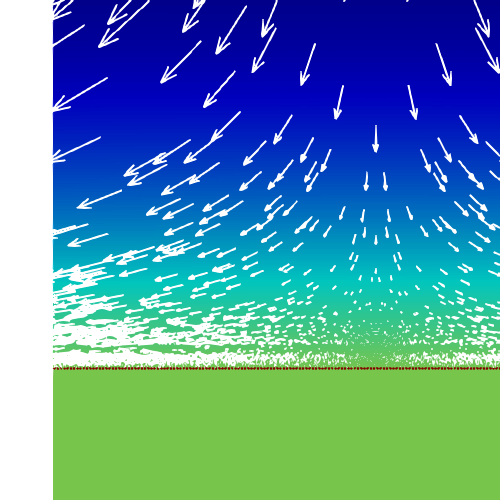}
		\caption{$\e=0.005$}
	\end{subfigure} 
	\hfill
	\begin{subfigure}[b]{0.25\textwidth}
		\centering
		\includegraphics[width=\textwidth]{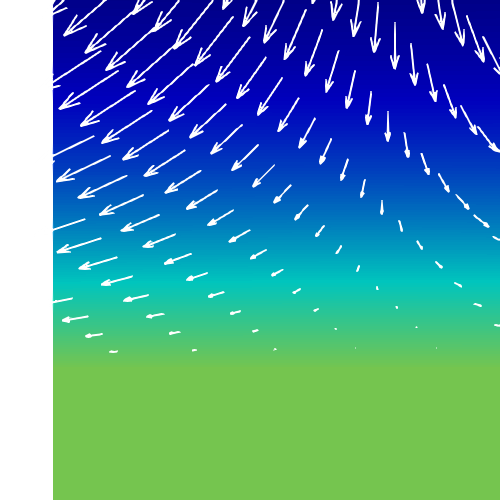}
		\caption{Effective solution}
	\end{subfigure} 
	\hfill
	\begin{subfigure}[b]{0.05\textwidth}
		\hspace{-10pt}
		\includegraphics[width=\textwidth]{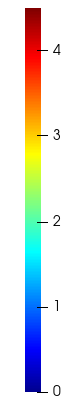}
	\end{subfigure}
	\caption{The solution for the two-dimensional setup. Visible are the fluid flow and the two results for different $\e$ values as well as the effective fluid and solid temperature. For better readability, we only show the zoomed-in section $[0, 0.7] \times [-0.2, 0.55]$.}
	\label{fig:temperature_flow_2d}
\end{figure}

\begin{figure}[H]
	\centering
	\begin{subfigure}[b]{0.43\textwidth}
		\centering
		\includegraphics[width=\textwidth]{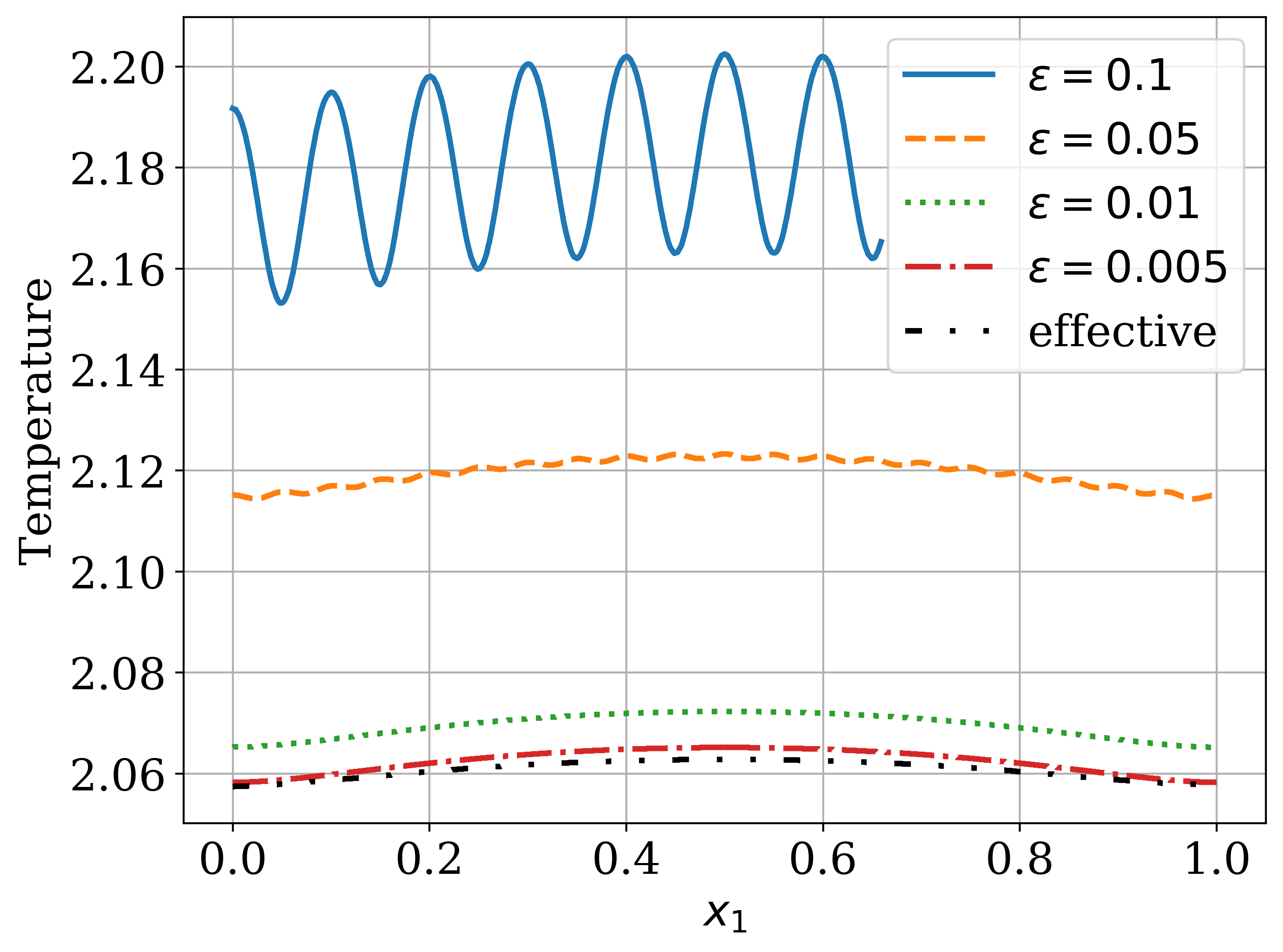}
	\end{subfigure}
	\hfill
	\begin{subfigure}[b]{0.28\textwidth}
		\centering
		\begin{subfigure}[b]{0.45\textwidth}
			\centering
			\includegraphics[width=\textwidth]{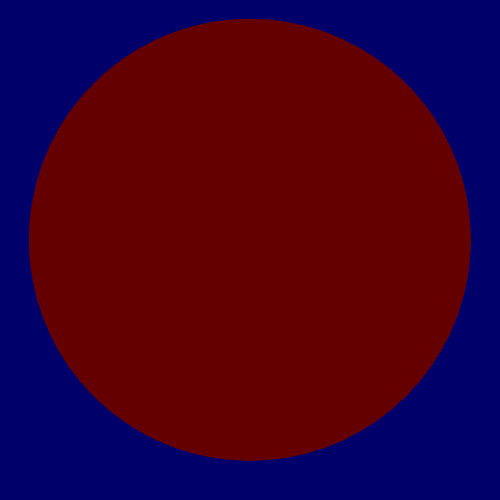}
			\caption{$\e=0.1$}
		\end{subfigure}
		\begin{subfigure}[b]{0.45\textwidth}
			\centering
			\includegraphics[width=\textwidth]{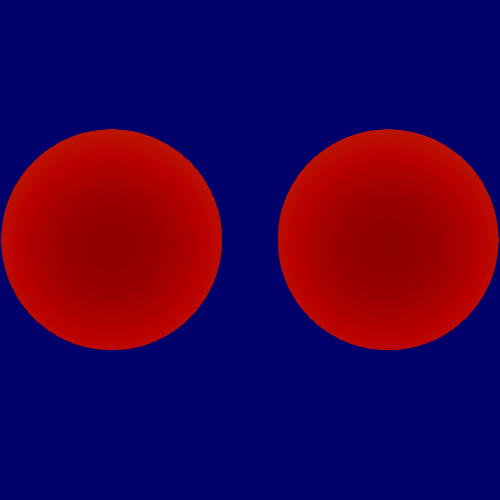}
			\caption{$\e=0.05$}
		\end{subfigure}
		\\
		\begin{subfigure}[b]{0.45\textwidth}
			\centering
			\includegraphics[width=\textwidth]{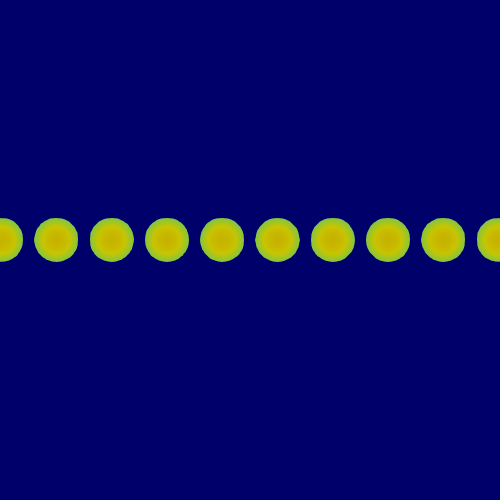}
			\caption{$\e=0.01$}
		\end{subfigure}
		\begin{subfigure}[b]{0.45\textwidth}
			\centering
			\includegraphics[width=\textwidth]{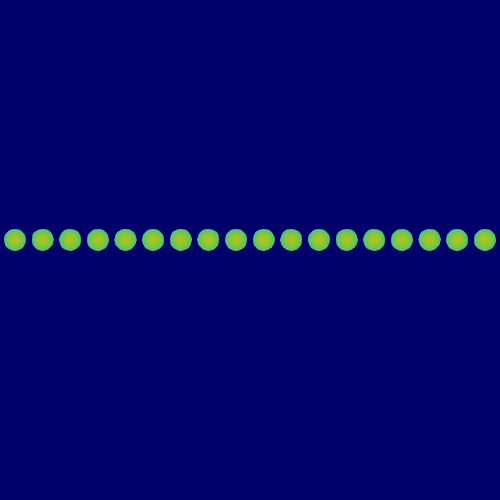}
			\caption{$\e=0.005$}
		\end{subfigure}
	\end{subfigure} 
	\begin{subfigure}[b]{0.15\textwidth}
		\centering
		\begin{subfigure}[b]{0.95\textwidth}
			\centering
			\includegraphics[width=\textwidth]{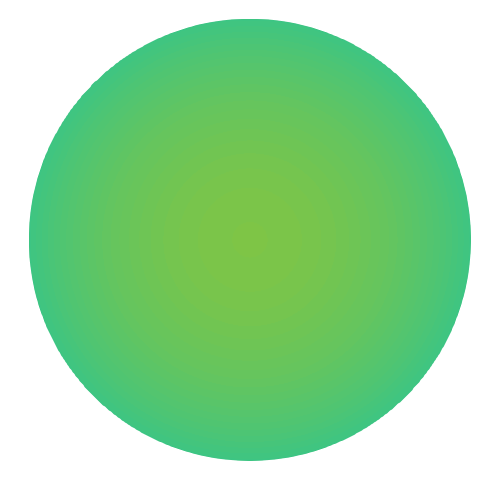}
			\caption{Effective cell temperature}
		\end{subfigure}
		\vspace{35pt}
	\end{subfigure} 
	\hfill
	\begin{subfigure}[b]{0.05\textwidth}
		\hspace{-21pt}
		\includegraphics[width=1.3\textwidth]{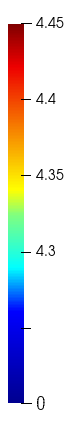}
	\end{subfigure}
	\caption{Convergence study for the two-dimensional setup. Left: solutions along a horizontal line at $x_2=0.05$, for different $\e$. Right: a zoomed-in section of the cell temperature at the left boundary. Note the different temperature scale, compared to Figure  \ref{fig:temperature_flow_2d}a--d show the solution of the resolved microscale case and (e) the effective temperature $\theta^g(\Tilde{x}, y)$ at the point $\Tilde{x}=(0,0)$.}
	\label{fig:temperature_profile_2d}
\end{figure}

The simulation results for the two-dimensional case are shown in Figures  \ref{fig:temperature_flow_2d},\ref{fig:temperature_profile_2d}. The flow field and macroscopic temperature are demonstrated in Figure \ref{fig:temperature_flow_2d}.
We observe that the system becomes cooler for smaller $\e$ since the fluid is slowed down less by the grains, and the cooling due to convection becomes stronger.
Given the underlying flow field, we obtain a solution that is symmetrical around the center of the domain.
Similar to the results above, in Figure \ref{fig:temperature_profile_2d} we can again observe the convergence of the microscale solution to the effective solution. Besides the behavior of the macroscopic solution, the right side of Figure \ref{fig:temperature_profile_2d} also demonstrates the grain temperature.
Here, the temperature scale is different than in Figure \ref{fig:temperature_flow_2d} to better investigate the grain temperature.
The grains also become slightly cooler for smaller $\e$. In the center of the grains, the temperature is higher since the heat source is inside the grains. This aspect is also captured in the effective model. 

\section{Conclusion}
We studied the effective influence of grain structures located on an interface between a fluid and a solid by the use of two-scale convergence for thin domains.
Two distinct scenarios were considered: Case (a) with disconnected grains and Case (b) with a connected grain structure.
For Case (a), we derived an effective two-scale model with microstructures at the interface.
In Case (b), we obtained, next to the temperature of fluid and solid, an effective interface temperature for the grains in the homogenized model.
The homogenization results were verified by direct comparison with the microscale model with the help of numerical simulations.
To this end, we considered an iterative algorithm to realize the coupling of grains and macro temperature. We showed numerically that the iteration speed can be improved by utilizing a relaxation scheme.
The numerical results support the derived effective model and demonstrate that it can accurately capture the behavior inside the grains.

For further studies, it would be interesting to investigate the influence of the Assumption \ref{item:A1}, where a specific $\e$-scaling of the heat conductivity inside grains was chosen.
Additionally, the assumption for the trace estimate could be weakened such that the derived model could be applied to general Lipschitz domains.

Regarding our motivating application of the grinding process, we could determine a first effective model that could be used to include the grains in further simulations. Of course, we are currently assuming periodic abrasive grains, which does not reflect reality.
However, with the help of numerical approaches, the derived models could be transferred to the more general non-periodic case. In addition, the interaction with the workpiece has been disregarded but plays an important role.
Of particular interest here would be the homogenization of the grinding gap, for which a flow equation must also be considered.
The extension employing numerical methods and the consideration of the grinding gap will be investigated in further research.
\section*{Use of AI tools declaration}
The authors declare that they have not used Artificial Intelligence (AI) tools in the creation of this article.

\section*{Acknowledgements}
The authors are grateful to the anonymous reviewers for their constructive feedback and helpful suggestions which significantly improved the initial manuscript.

TF acknowledges funding by the Deutsche Forschungsgemeinschaft (DFG, German Research Foundation) -- project nr.  
281474342/GRK2224/2. Additionally, this research was also partially funded by the DFG -- project nr. 439916647.

The research activity of ME is funded by the European Union’s Horizon 2022 research and innovation program under the Marie Skłodowska-Curie fellowship project {\textit{MATT}} (\url{https://doi.org/10.3030/101061956}).
\section*{Conflict of interest}

The authors declare that there is no conflict of interest.
\bibliographystyle{abbrv}
\bibliography{MainArxiv.bbl}

\renewcommand{\thefigure}{A\arabic{figure}}
\setcounter{figure}{0}
\begin{appendices}
\section{Convergence study with respect to mesh resolution}
\begin{figure}[H]
    \centering
     \begin{subfigure}[b]{0.40\textwidth}
         \centering
         \includegraphics[width=\textwidth]{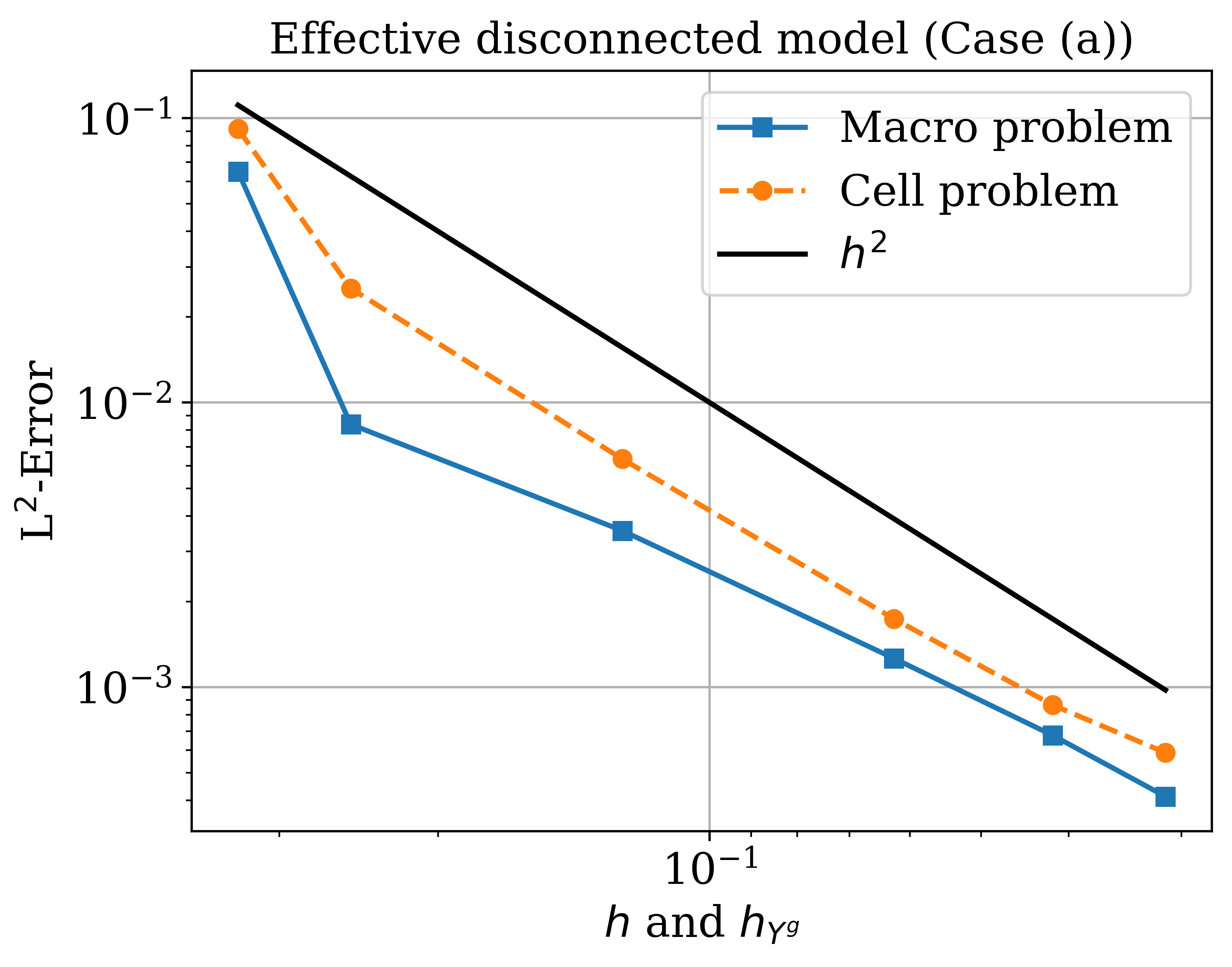}
     \end{subfigure}
     \hfill
     \begin{subfigure}[b]{0.40\textwidth}
         \centering
         \includegraphics[width=\textwidth]{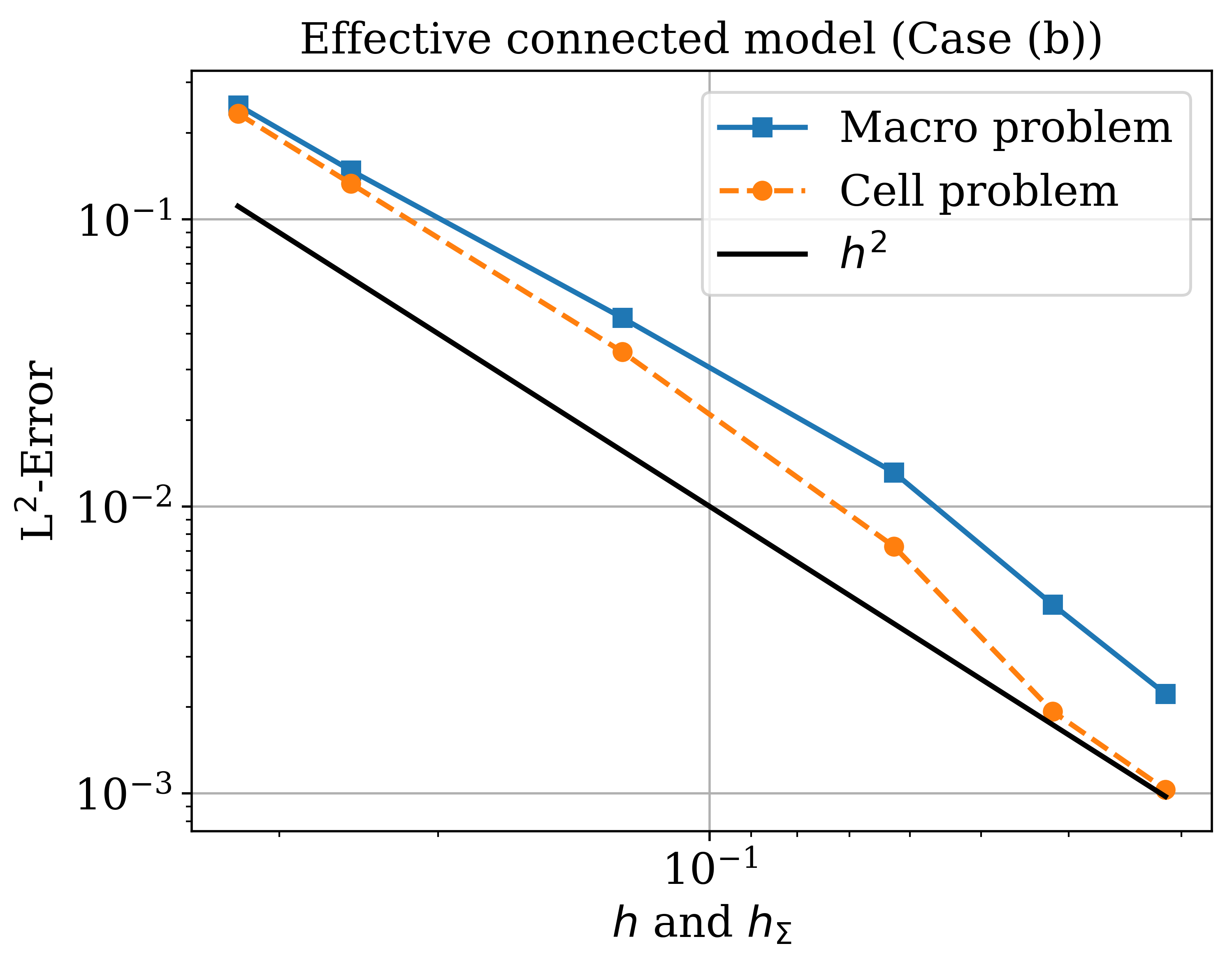}
     \end{subfigure}
     \\
     \begin{subfigure}[b]{0.40\textwidth}
         \centering
         \includegraphics[width=\textwidth]{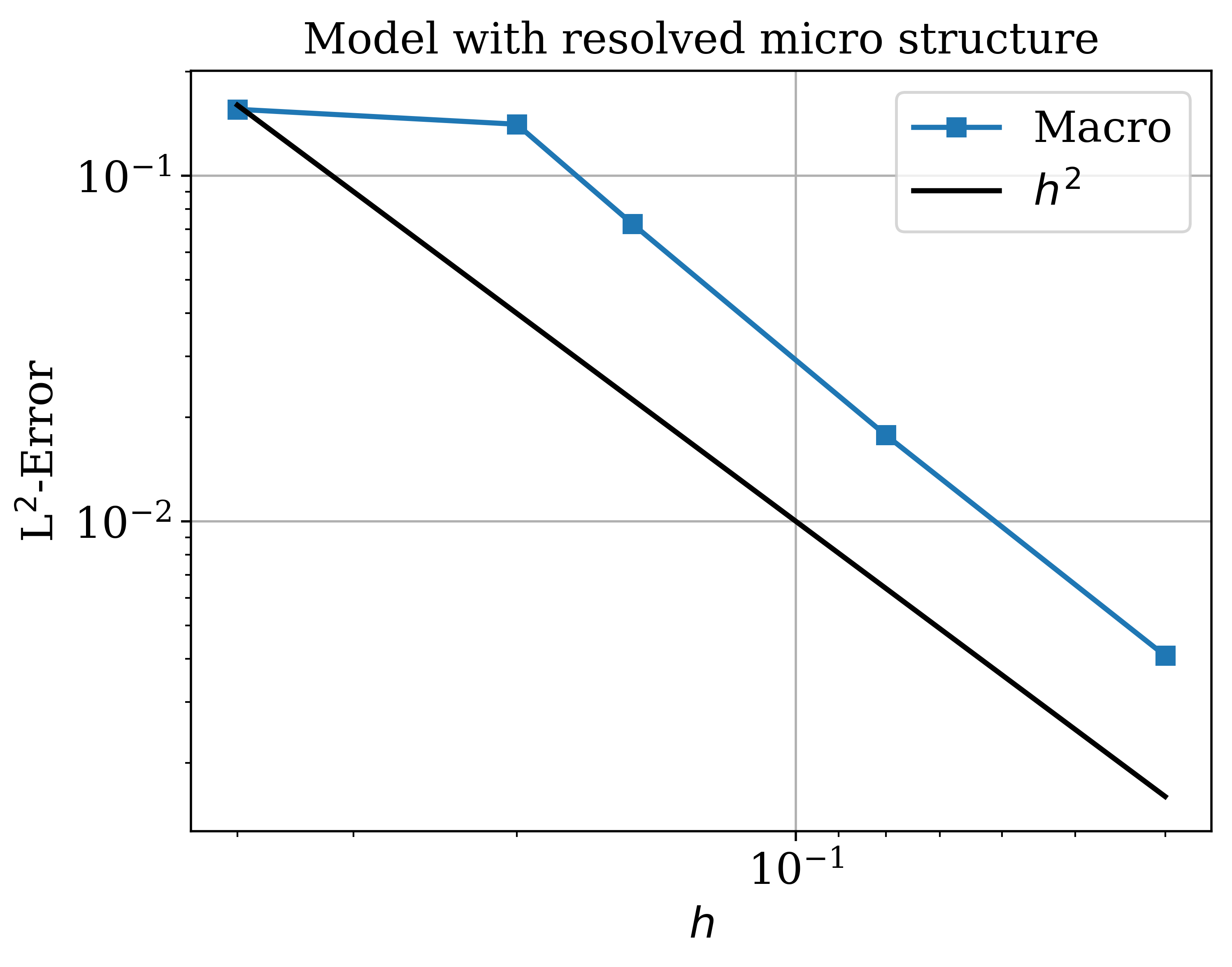}
     \end{subfigure}
     \caption{Convergence study for the simulated temperature $\theta$ with respect to the mesh resolution. The problem with resolved microstructure was solved for $\varepsilon=0.1$ and the microstructures were locally refined with a resolution of $h_\varepsilon = 2\varepsilon h$. All simulation results were compared to a simulation with resolution $h=0.015$.}
     \label{fig:convergence_for_mesh_resolution}
\end{figure}
\end{appendices}
%

\end{document}